\documentclass[10pt,%                       
               a4paper,%                    
               oneside,%                    
			   reqno%								% numbering of equations on the right									
               ]{amsart}
                                          
\usepackage{amsmath,amsthm,amssymb}
\usepackage{physics}

\usepackage{xy}
\xyoption{all}

\usepackage[T1]{fontenc}
\usepackage{lmodern}
\usepackage[utf8]{inputenc}				%

\usepackage[english]{babel}

\usepackage[dvipsnames]{xcolor}

\usepackage{csquotes}
\usepackage{bookmark}

\usepackage{microtype}
\usepackage{hyperref}

\newcommand{\numberset}{\mathbb} 
\newcommand{\N}{\numberset{N}} 
\newcommand{\Z}{\numberset{Z}} 
 
\newcommand{\R}{\numberset{R}}
\newcommand{\I}{\numberset{I}}
\let\C\relax
\newcommand{\C}{\numberset{C}}

\DeclareMathOperator{\KKK}{KK}
\DeclareDocumentCommand{\KK}{O{} m m m}{\KKK_{#1}^{#2}(#3,#4)}
\DeclareDocumentCommand{\RKK}{O{} m m m}{\mathrm{RKK}_{#1}^{#2}(\E;#3,#4)}
\DeclareDocumentCommand{\KKe}{O{*} O{} m m}{\KKK_{#1}^{#2}(#3,#4)}
\newcommand{\hatotimes}{\,\widehat{\otimes}\,}
\newcommand{\hot}{\widehat{\otimes}}

\let\E\relax
\let\Res\relax
\newcommand{\E}{\underline{E}G}

\DeclareMathOperator{\Res}{Res}

\newcommand{\Calg}{C^\ast\text{-algebra}}

\theoremstyle{plain}
\newtheorem{theorem}{Theorem}

\newtheorem{theoremA}{Theorem}

\newtheorem{corollaryA}[theoremA]{Corollary}

\newtheorem*{theorem*}{Theorem}
\newtheorem*{corollary*}{Corollary}
\newtheorem*{assumption*}{Assumption}
\newtheorem*{question*}{Question}
\newtheorem{proposition}[theorem]{Proposition}
\newtheorem{lemma}[theorem]{Lemma}
\newtheorem{remark}[theorem]{Remark}
\newtheorem{corollary}[theorem]{Corollary}

\theoremstyle{definition} 
\newtheorem{definition}[theorem]{Definition}

\theoremstyle{plain}

\theoremstyle{definition}

\DeclareRobustCommand{\SkipTocEntry}[5]{}

\begin{document}
%******************************************************************
% Beginning
%******************************************************************
%\input{BeginEnd/Abstract}

%\phantomsection
%\pdfbookmark[1]{Sommario}{Sommario}
\begin{abstract}
Building on work by Kasparov, we study the notion of Spanier--Whitehead $K$-duality for a discrete group. It is defined as duality in the $\KKK$-category between two $C^*$-algebras which are naturally attached to the group, namely the reduced group $C^*$-algebra and the crossed product for the group action on the universal example for proper actions. %$C^*_r(G)$ and $C_0(\E)\rtimes G$.
We compare this notion to the Baum--Connes conjecture by constructing duality classes based on two methods: the standard ``gamma element'' technique, and the more recent approach via cycles with property gamma. As a result of our analysis, we prove Spanier--Whitehead duality for a large class of groups, including Bieberbach's space groups, groups acting on trees, and lattices in Lorentz groups.
\end{abstract}

\title[Groups with Spanier--Whitehead duality]{Groups with Spanier--Whitehead duality}

\author{Shintaro Nishikawa}

\address{Department of Mathematics, Penn State University, University Park, PA 16802, USA}
\email{sxn28@psu.edu}

\author{Valerio Proietti}
\address{Research Center for Operator Algebras, East China Normal University, 3663 ZhongShan North Road, Putuo District, Shanghai 200062, China}
\email{proiettivalerio@math.ecnu.edu.cn}

%\thanks{Thanks.}

\subjclass[2010]{46L85; 46L80, 55P25}
\keywords{Spanier--Whitehead duality, Poincar\'e duality, Baum--Connes conjecture, direct splitting method, noncommutative topology.}

%\begin{nouppercase}
\maketitle
%\end{nouppercase}

%\input{BeginEnd/Index}

%\pdfbookmark[1]{\contentsname}{tableofcontents}
\setcounter{tocdepth}{2}
\tableofcontents
%\markboth{\contentsname}{\contentsname} 

%*******************************************************
% Figures
%*******************************************************    
%\phantomsection
%\pdfbookmark[1]{\listfigurename}{lof}
%\listoffigures

%*******************************************************
% Tables
%*******************************************************
%\phantomsection
%\pdfbookmark[1]{\listtablename}{lot}
%\listoftables

%******************************************************************
% Main
%******************************************************************
%\input{Sections/intro}

%************************************************
\addtocontents{toc}{\SkipTocEntry}\section*{Introduction}
%************************************************

Alexander duality applies to the homology theory properties of the complement of a subspace inside a sphere in Euclidean space. More precisely, for a finite complex $X$ contained in $S^{n+1}$, if $\tilde{H}$ denotes reduced homology or cohomology with coefficients in a given abelian group, there is an isomorphism $\tilde{H}_i(X) \cong \tilde{H}^{n-i}(S^{n+1}\setminus X)$, induced by slant product with the pullback of the generator $\mu^*([S^n])$, via the duality map $\mu\colon X\times (S^{n+1}\smallsetminus X)\to S^{n}$, $\mu(x,y)=(x-y)/\norm{x-y}$.

Ed Spanier and J.~H.~C. Whitehead generalized this statement and adapted it to the context of stable homotopy theory. Their basic intuition was that sphere complements determine the homology, but not the homotopy type, in general. However, the stable homotopy type can be deduced and provides a first approximation to homotopy type  \cite{spwh:duality}. Thus, the modern statement is phrased in terms of dual objects $X,DX$ in the category of pointed spectra with the smash product as a monoidal structure, and by taking maps to an Eilenberg-MacLane spectrum one recovers Alexander duality formally. 

The modern version of the duality implies Poincar\'e duality for compact manifolds and extends in a natural way to generalized cohomology theories such as $K$-theory. In this setting, a compact $\text{spin}^c$-manifold exhibits Poincar\'e duality in the sense that the $K$-homology class of the Dirac operator induces by cap product an isomorphism $K^*(M)\to K_{*+n}(M)$, where the shift is given by the dimension \cite{kas:descent}. 

More generally, the bivariant version of $K$-theory introduced by Kasparov, which we shall use extensively in this paper, showcases a close relationship to Alexander-Spanier duality; by this we mean that for $X,Y$ finite complexes one has a chain of isomorphisms (\cite{kamscho:swd})
\[
\KKe{C(X)}{C(Y)}\cong \KKe{\C}{C(DX\wedge Y)}\cong K_*(C(DX\wedge Y))\cong K^*(DX\wedge Y).
\]
Having introduced $C^*$-algebras in this way, as they arise naturally in applications to topology, dynamics, and index theory, and are generally noncommutative, it is natural to seek for generalizations of Spanier--Whitehead duality in the framework of noncommutative geometry.

For a separable, nuclear $C^*$-algebra $A$ represented on a Hilbert space, the commutant of its projection into the Calkin algebra has some of the properties reminiscent of a Spanier--Whitehead $K$-dual. This is the \emph{Paschke dual} of $A$, and satisfies $K_*(P(A))\cong K^*(A)$. However, in general $P(A)$ is neither separable nor nuclear, the Kasparov product is not defined, so that it seems desirable to explore different routes for the definition of a $K$-dual.

In \cite{connes:ncg} A.~Connes introduced the appropriate formalism for this question, which shall be described shortly, and in \cite{connes:gravity} he showed the first nontrivial example of a noncommutative Poincar\'e duality algebra, in the form of the irrational rotation algebra. In \cite{eme:poinhyp} H.~Emerson proves the same result for the crossed product of a hyperbolic group acting on its Gromov boundary. Examples of pairs of algebras with Spanier--Whitehead duality were also given by Kaminker and Putnam \cite{kamput:dualshifts} in the case of Cuntz--Krieger algebras associated respectively to $M$ and its transpose, where $M$ is a square $\{0,1\}$-valued matrix. Their result is a special case of a more general one, in which the stable and unstable Ruelle algebras of a Smale space are shown to be in duality \cite{wkp:ruellepoinc}. Duality in $K$-theory also appears in connection with string theory on noncommutative spacetimes \cite{bvrs:string,bvrs:string2}. %Finally in \cite{meyereme:dualities} the authors explore other kinds of dualities in the equivariant setting, not unrelated to Spanier--Whitehead duality.

In this paper, $G$ is a discrete group which admits a $G$-compact model $\E$ of the classifying space for proper actions \cite{BCH:class}. We study the question of Spanier--Whitehead duality for the pair of $C^*$-algebras $C^*_r(G)$ and $C_0(\E)\rtimes G$, where the latter is the crossed product for the group action on $\E$. 

This problem is tightly related to the Baum--Connes conjecture and in particular to the so-called Dirac dual-Dirac method. This goes back to the seminal work of Kasparov \cite[Sections 4 and 6]{kas:descent} and is further explored in \cite[Section 6]{kasskand:buildnov}. In a different direction, the relationship between the assembly map and Fourier-Mukai duality is discussed in \cite{block:dual}.

The idea of an underlying noncommutative duality whenever Dirac and dual-Dirac classes are available is well-known to experts (see for example \cite[Example 2.14]{bvrs:string} and \cite[Theorems 2.9 and 3.1]{echemekim:dualtwisted}). In particular the work \cite{meyereme:dualities} by Emerson and Meyer shares many ideas with the present paper, while working in the context of equivariant $\KKK$-theory and groupoids. See Subsection \ref{subsec:ekd} and Remark \ref{rem_Kasparovdual} for more details.

Below are two main results of this paper. More details on statements and terminology are given in the sequel.

\begin{theorem*} Suppose the $\gamma$-element exists. Then $C_0(\E)\rtimes G$ is a Spanier--Whitehead $K$-dual of $C^*_r(G)$ (in a canonical way) if and only if $G$ satisfies the strong Baum--Connes conjecture.
\end{theorem*}

\begin{corollary*} For all a-T-menable groups $G$ which admit a $G$-compact model of $\E$, $C_0(\E)\rtimes G$ is a Spanier--Whitehead $K$-dual of $C^*_r(G)$.
\end{corollary*}

% (note that since the $G$-action is amenable the full and reduced crossed products coincide \cite{claire:amenexact}).

\addtocontents{toc}{\SkipTocEntry}\section*{Acknowledgements}

The first author would like to thank  H.~Emerson for his helpful comments and suggestions on this topic. He was supported by Prof. N.~Higson during the summer 2019 at Penn State University. He would like to thank him for the support.

The second author would like to thank B.~Mesland and R.~Nest for stimulating conversations during the early stages of this paper. This research was partly supported through V.~Proietti's ``Oberwolfach Leibniz Fellowship'' by the \emph{Mathematisches Forschungsinstitut Oberwolfach} in 2019. In addition, the second author was partly supported by the Science and Technology Commission of Shanghai Municipality (STCSM), grant no. 13dz2260400.  

We would like to thank K.~Li for his helpful comments and suggestions on an earlier version of this paper.
\bookmarksetup{startatroot}
\subsection{Noncommutative Spanier--Whitehead duality}

Let us see the main notions we will be working with. In what follows the $C^*$-tensor product is understood to be spatial.

\begin{definition}\label{def:wpoin} (cf. {\cite[Section 2.7]{bvrs:string}})
Let $A,B$ be separable $C^*$-algebras. We say that $B$ is a \emph{weak Spanier--Whitehead $K$-dual} of $A$ if there are elements
\begin{align*}
d&\in \KKe[i]{A\otimes B}{\C}\\
\delta &\in \KKe[-i]{\C}{A\otimes B}
\end{align*}
such that the induced maps
\begin{align*}
d_j\colon K_j(A)&\to K^{j+i}(B)&  d_j(x)&=x\,\hot_A\, d\\
\delta_j\colon K^j(B)&\to K_{j-i}(A)&  \delta_j(x)&=\delta\,\hot_B\, x
\end{align*}
are isomorphisms and inverses to each other.
\end{definition}

Note that, unlike the case of topological spaces, in the noncommutative context the existence of $d$, given $\delta$, is an additional requirement.

Some notation: $1_A\in \KKe[0]{A}{A}$ stands for the ring unit, $\sigma\colon A\otimes B \cong B\otimes A$ denotes the flip isomorphism. Recall as well the homomorphism $\tau_B\colon\KKe{A}{A}\to\KKe{A\otimes B}{A\otimes B}$, given on cycles as
\[
(\phi, H, T)\mapsto (\phi\hatotimes 1, H\hatotimes B, T \hatotimes 1),
\]
and equally defined via Kasparov product (over the complex numbers) by $\tau_B(x)=x\hatotimes 1_B=1_B\hatotimes x$.

\begin{lemma}[{\cite[Lemma 9]{eme:poinhyp}}]\label{lem:spoin}
In the setting of Definition \ref{def:wpoin}, we have the following identities:
\begin{align*}
(\delta_{j+i}\circ d_j)(y)&=(-1)^{ij}y\,\hot_{A}\,\Lambda_A\\
(d_{j-i}\circ \delta_j)(y)&=(-1)^{ij}\Lambda_B\,\hot_{B}\,y,
\end{align*}
where the elements $\Lambda_A\in \KKe[0]{A}{A}$ and $\Lambda_B\in \KKe[0]{B}{B}$ are defined as
\begin{align*}
\Lambda_A&=\delta\,\hot_B\, d =(\delta\hatotimes 1_A) \,\hot_{A\otimes B\otimes A}\, (1_A \hatotimes \sigma^*(d))\\
\Lambda_B&=\delta\,\hot_A\, d =(\sigma_*(\delta)\hatotimes 1_B) \,\hot_{B\otimes A\otimes B}\, (1_B \hatotimes d).
\end{align*}
\end{lemma}
%\begin{proof}
%We verify the first equation, the second follows similarly (see \cite[Lemma 9]{eme:poinhyp}). From the definition we get
%\begin{align*}
%(\delta_{j+i}\circ d_j)(y)&=\delta\,\hot_{A\otimes B}\, (1_A \hatotimes ((y\hatotimes 1_B)\,\hot_{A\otimes B}\, d))\\
%&=\delta\,\hot_{A\otimes B}\,(1_A \hatotimes y \hatotimes 1_B)\,\hot_{A\otimes A\otimes B} (1_A \hatotimes d)\\
%&=\delta\,\hot_{A\otimes B}\,(1_A \hatotimes \sigma_*(y \hatotimes 1_B))\,\hot_{A\otimes B\otimes A} (1_A \hatotimes \sigma^*(d))\\
%&=(-1)^{ij}y\,\hot_A\,(\delta\hatotimes 1_A)\,\hot_{A\otimes B\otimes A} (1_A \hatotimes \sigma^*(d))\\
%&=(-1)^{ij}y\,\hot_A\,\Lambda_A,
%\end{align*}
%where the second-last equality follows by graded commutativity of the general form of the Kasparov product.
%\end{proof}

\begin{definition}\label{def:swd}% (cf. {\cite[Definition 4.1]{wkp:ruellepoinc}})
Let $A,B$ denote $C^*$-algebras in weak Spanier--Whitehead duality. With notation from Lemma \ref{lem:spoin}, if we have $\Lambda_A=1_A$ and $\Lambda_B=(-1)^i 1_B$, we say that $A$ and $B$ satisfy \emph{Spanier--Whitehead $K$-duality}.
\end{definition}

Note that this definition is symmetric, so that it can equivalently be phrased by saying that $B$ is a Spanier--Whitehead $K$-dual of $A$, in alignment with the weak form introduced earlier.

\begin{remark}\label{rem:dualizable}
In the tensor category $(\KKK,\otimes)$, where objects are $C^*$-algebras and $\mathrm{Hom}(A,B)=\KKe[0]{A}{B}$, the previous definition (for $i=j=0$) can be reinterpreted as the statement that {$A$ is a dualizable object and $B$ is its dual. In other words the following triangle identity (and its analogue swapping $A$ and $B$) holds
\[
\xymatrix{& A\otimes B \otimes A \ar[rd]^-{d\hatotimes 1_A} &\\
A \ar[ru]^-{1_A\hatotimes \delta}  \ar[rr]^-{1_A} && A}
\]
up to the unique isomorphisms coming from braiding and $A\otimes \C\cong A$.}
%$\tau_A$ is left adjoint to $\tau_B$. The component at $X$ of the counit is given by $\tau_X(d)$, while the component at $X$ of the unit is given by $\tau_X(\delta)$. The conditions on $\Lambda_A$ and $\Lambda_B$ translate precisely to the counit-unit equations.
\end{remark}

The Spanier--Whitehead $K$-dual respects tensor products in the following sense: if the dual of $A$ is $B$ and the dual of $A^\prime$ is $B^\prime$, then the dual of $A\otimes B$ is $\KKK$-equivalent to $A^\prime\otimes B^\prime$, provided it exists (see \cite{kamscho:swd}).

\medskip

Throughout this paper $G$ denotes a countable discrete group admitting a $G$-compact model for its universal example for proper actions.

\begin{definition}
We say that $G$ \emph{has (weak) Spanier--Whitehead $K$-duality} if $C_0(\E)\rtimes G$ is a (weak) dual of $C^*_r(G)$.
\end{definition}

\begin{remark}\label{r:cross}
It follows from \cite[Proposition 2.2]{claire:amenexact} that the action of $G$ on $\E$ is amenable. Then by \cite[Theorem 5.3]{claire:amenexact} the associated full and reduced crossed products are isomorphic. In particular, any covariant pair of representations for $C_0(\E)$ and $G$ gives rise to a representation of the reduced crossed product $C_0(\E)\rtimes G$, namely the integrated form.
\end{remark}

In short, the aim of this paper is identifying an element $x$ belonging to the ``representation ring'' $\KKK^G_0(\C, \C)$, and constructing classes $d$ and $\delta$ as above in such a way that $\Lambda_{C^*_r(G)}$ and $\Lambda_{C_0(\E)\rtimes G}$ are both expressible in terms of $x$. Then the sought duality is reduced to studying the homotopy class of such element.

\subsection{Baum--Connes conjecture: the duality perspective}
The Baum--Connes conjecture \cite{BCH:class} states that the Baum--Connes assembly map
\begin{equation}\label{eq:bcbc}
\mu^G\colon \KKK^G_\ast(C_0(\E), \mathbb{C}) \to \KKK_\ast(\mathbb{C}, C^\ast_r(G))
\end{equation}
is an isomorphism of abelian groups. A generalization ``with coefficients'' can be introduced by inserting a $G$-algebra $A$ in the right ``slot'' of the left-hand side of \eqref{eq:bcbc}, and by considering the corresponding reduced crossed product in the target group:
\begin{equation}\label{eq:assmapcoeffintro}
\mu^G_A\colon \KKK^G_\ast(C_0(\E), A) \longrightarrow \KK{}{\C}{A\rtimes_r G}.
\end{equation}

Going back to the case with trivial coefficients (i.e., $A=\C$), since $G$ is a discrete group, the (dual) Green--Julg isomorphism (\cite{black:kth,valjens:misb,markus:ass}) 
\[
 \KKK_\ast^G(C_0(\E), \mathbb{C}) \cong \KKK_\ast(C_0(\E)\rtimes_rG, \mathbb{C})
\]
allows us to view the assembly map as a morphism
\begin{equation}\label{eq:assmapintro}
\KKK_\ast(C_0(\E)\rtimes G, \C) \longrightarrow \KKK_\ast(\C, C^*_r(G)).
\end{equation}
We shall see that this map is given by Kasparov product with a certain element 
\[
\delta\in \KK{}{\C}{C^*_r(G)\otimes C_0(\E)\rtimes G}
\]
(see Definition \ref{def:cunit}). Thus, the Baum--Connes conjecture for a discrete group $G$ admitting a $G$-compact model $\E$ is equivalent to the assertion that the element $\delta$ induces the isomorphism
\[
\delta_\ast\colon K^\ast(C_0(\E)\rtimes G)  \xrightarrow{\cong}  K_\ast(C^*_r(G)).
\]
A priori, this isomorphism itself is not enough to conclude that $G$ has weak Spanier--Whitehead $K$-duality. In this paper, under an assumption (see below), we identify an element
\[
d \in \KKK(C^*_r(G)\otimes C_0(\E)\rtimes G, \mathbb{C})
\]
which induces a map
\[
d_\ast\colon  K_\ast(C^*_r(G))  \xrightarrow{} K^\ast(C_0(\E)\rtimes G)
\]
which is the inverse of $\delta_\ast$ in favorable circumstances, namely if the Baum--Connes conjecture holds (it is a left inverse in general). Our assumption for constructing such an element $d$ is that the existence of the so-called gamma element, or alternatively the $(\gamma)$-element for $G$. Let us briefly review these notions.

\subsection{The \texorpdfstring{$\gamma$}{gamma}-element and the \texorpdfstring{$(\gamma)$}{(gamma)}-element}\label{subsec:ddd}
The following notion of the gamma element originates in Kasparov's work \cite{kas:descent}.
\begin{definition}\label{def:gamma}(See \cite{tu:bcgroupoid})
An element $x$ in $\KKK^G(\mathbb{C}, \mathbb{C})$ is called a \emph{gamma element} for $G$ if it satisfies the following:
\begin{enumerate}
\item for any finite subgroup $F\subseteq G$, we have
\[
\Res_G^F(x)=1_\C\in \KK{F}{\C}{\C}.
\]
\item for some separable, proper $G$-$C^\ast$-algebra $P$, we have
\[
x=\beta\,\hot_P\, \alpha \quad\,\text{where}\,\, \alpha\in \KKK^G(P, \mathbb{C}), \,\, \beta \in \KKK^G(\mathbb{C}, P).
\]
\end{enumerate}
\end{definition}
A gamma element for $G$, if it exists, is a unique idempotent in $\KKK^G(\mathbb{C}, \mathbb{C})$ which is characterized by the listed properties. Thus, we call it the gamma element for $G$ and denote it by $\gamma$. The existence of the gamma element for $G$ implies that the Baum--Connes assembly map is split-injective for all coefficients $A$ (cf. \cite{tu:bcgroupoid}), and furthermore that the assembly map $\mu^G_A$ is surjective if and only if $\gamma$ acts as the identity on $K_\ast(A\rtimes_rG)$ via ring homomorphisms
\begin{equation} \label{ringhoms}
\KKK^G(\mathbb{C}, \mathbb{C}) \to \KKK^G(A, A) \to \KKK(A\rtimes_rG, A\rtimes_rG) \to \mathrm{End}(K_\ast(A\rtimes_rG)).
\end{equation} 
The other composition $y=\alpha \hatotimes \beta$ is an idempotent in $\KKK(P, P)$ which may not be the identity on $P$ in general. Upon replacing $P$ with its ``summand'' $P_{\C}=yP$, which can be defined as a limit of $P \xrightarrow{y} P \xrightarrow{y} \cdots$ in the category $\KKK^G$ (cf. \cite[Proposition 1.6.8]{nee:tri}), we can arrange $\alpha$ (and $\beta$) above to be a weak-equivalence, meaning that $\Res_G^F(\alpha)$ is an isomorphism for any finite subgroup $F$ of $G$. In this case, the element $\alpha$ in $\KKK^G(P_{\C}, \mathbb{C})$ is called the \emph{Dirac element} and can be characterized up to equivalence by the fact that $\alpha$ is a weak-equivalence from a ``proper object'' $P_\C$ to $\mathbb{C}$. Meyer and Nest \cite{nestmeyer:loc} show that the Dirac element always exists for any group $G$ but, in general, it is not known whether $P_\C$ can be taken to be a proper $C^*$-algebra. For most of the known examples,  $P_\C$ can indeed be assumed to be proper, meaning that we may think $P=P_\C$. However, we emphasize that the algebra $P$ appearing in the definition can be quite arbitrary whereas $P_\C$ is a uniquely characterized object.

In \cite{nish:dsplitb}, the first author introduced a notion called the $(\gamma)$-element, which can be thought of as an alternative description of the gamma element, bypassing the necessity of a proper algebra $P$ for its definition. 

Recall that we assume that $G$ admits a $G$-compact model for $\E$. We use $[\cdot,\cdot]$ to denote the commutator.% We call a function $c\in C_c(\E)$ a cutoff function of $\E$ if $\sum_{g\in G} g(c)^2=1$.

\begin{definition}(\cite[Definition 2.2]{nish:dsplitb})
A cycle $(H,T)$ representing an element $[H, T]$ in $\KK{G}{\C}{\C}$ is said to have \emph{property $(\gamma)$} if it satisfies the following:
\begin{enumerate}
\item for any finite subgroup $F\subseteq G$ we have
\[
\Res_G^F([H,T])=1_\C\in \KK{F}{\C}{\C}.
\]
\item there is a non-degenerate $G$-equivariant representation of $C_0(\E)$ on $H$ such that 
\begin{enumerate}
\item[(2.1)] the function
\[
g\mapsto [g\cdot f, T]
\]
belongs to $C_0(G,K(H))$, i.e., it vanishes at infinity and is compact-operator-valued for any $f\in C_0(\E)$.
\item[(2.2)] for some cutoff function $c\in C_c(\E)$ (i.e., $c$ is non-negative and satifies $\sum_{g\in G} g(c)^2=1$), we have
\[
T- \sum_{g\in G} (g\cdot c)T(g\cdot c)\in K(H).
\]
\end{enumerate}
\end{enumerate}
An element $x$ in $\KKK^G(\mathbb{C}, \mathbb{C})$ is called a ($\gamma$)-element for $G$ if it is represented by some cycle with property $(\gamma)$.
\end{definition}

A $(\gamma)$-element for $G$, if it exists, is a unique idempotent in $\KKK^G(\mathbb{C}, \mathbb{C})$ which is characterized by the listed properties. Thus, we call it the $(\gamma)$-element for $G$. If there is a gamma element $\gamma$ for $G$, there is a cycle with property $(\gamma)$ representing $\gamma$. Thus the two notions, the $\gamma$-element and the $(\gamma)$-element for $G$, coincide when $\gamma$ exists. The existence of the $(\gamma)$-element $x$ for $G$ implies that the Baum--Connes assembly map is split-injective for all coefficients $A$, and furthermore that the assembly map $\mu^G_A$ is surjective if and only if $x$ acts as the identity on $K_\ast(A\rtimes_rG)$ via ring homomorphisms \eqref{ringhoms}.

Given the existence of the $(\gamma)$-element, \cite{nish:dsplitb} introduced the so-called $(\gamma)$-morphism as a candidate for inverting the assembly map $\mu^G$. This is given by Kasparov product with a certain element
\[
\tilde x \in \KKK^G(C^\ast_r(G)\otimes C_0(\E), \C).
\]
The Green--Julg isomorphism allows us to get the corresponding element  $d\in \KKK(C^*_r(G)\otimes C_0(\E)\rtimes G, \C)$. 

Our proposed strategy aims at realizing weak Spanier--Whitehead duality through elements $\delta$ and $d$ respectively corresponding to the assembly map and the $(\gamma)$-morphism, which seems to be a natural situation. Furthermore, as a result of Lemma \ref{lem:spoin}, the surjectivity and injectivity of the assembly map are controlled respectively by $\Lambda_{C^*_r(G)}$ and $\Lambda_{C_0(\E)\rtimes G}$. This gives yet another interpretation of these two classes.

\subsection{Equivariant Kasparov duality}\label{subsec:ekd}

In \cite{meyereme:dualities} the authors study several duality isomorphisms between equivariant bivariant $K$-theory groups, generalizing Kasparov's first and second Poincar\'e duality isomorphisms. For many groupoids, both dualities apply to a universal proper $G$-space, which is the basis for the Dirac dual-Dirac method. In this setting they explain how to describe the Baum--Connes assembly map via localization of categories as in \cite{nestmeyer:loc}.

The main notion in \cite{meyereme:dualities} is that of \emph{equivariant Kasparov dual} for a $G$-space $X$. %We omit the precise definition here (see \cite[Definition 4.1]{meyereme:dualities}) and focus on the case where $X=\E$. This situation has an interesting overlap with Theorem \ref{thm:pdual} above, in the following sense.
It involves an $X \rtimes G$-$C^*$-algebra $P$, an element $\alpha\in\KK{G}{P}{\C}$, and an additional class $\Theta\in \mathrm{RKK}^G(X;\C,P)$ (see \cite[Definition 4.1]{meyereme:dualities} for more details). Recall that the category $\mathrm{RKK}^G(X)$ coincides with the range of the pullback functor $p^*_{X}\colon \KKK^G\to \KKK^{X\rtimes G}$ via the collapsing map $p\colon X\to \ast$. 

The case $X=\E$ is particularly relevant for our purposes. The class $\Theta$ may be thought as the ``inverse'' of $\alpha$ up to restriction to finite subgroups. More precisely, if a lifting $\beta\in\KK{G}{\C}{P}$ of $\Theta$ exists, then the axioms of equivariant Kasparov duality guarantee that $\beta\,\hot_P\,\alpha$ is the $\gamma$-element and $\alpha\,\hot_\C\,\beta=1_P$. In particular, we have $P=P_\C$ and $\alpha$ is a weak equivalence, and hence a Dirac morphism.

Let $Z$ denote the unit space of $G$ and suppose the moment map from $\E \to Z$ is proper. Then \cite[Theorem 5.7]{meyereme:dualities} establishes a connection to what we might call ``equivariant'' Spanier-Whitehead duality. We summarize it below for the convenience of the reader (see also Remark \ref{rem_Kasparovdual}).

\begin{theorem}
The triple $(P,\alpha,\Theta)$ is a Kasparov dual for $X$ if and only if $C_0(X)$ and $P$ are dual objects in $\KKK^G$ (cf. Remark \ref{rem:dualizable}) with duality unit and counit respectively induced by $\Theta$ and $\alpha$.
\end{theorem}

Concerning the connection with the Baum--Connes assembly map, we have:
\begin{theorem}[{\cite[Theorem 6.14]{meyereme:dualities}}]
Suppose $\E$ admits a local symmetric Kasparov dual. Then the assembly map $\mu^G_A$ is an isomorphism for all proper coefficient algebras $A$.
\end{theorem}

Assuming $\E$ to be $G$-compact, the proof of the previous theorem roughly goes as follows: the second Poincar\'e duality isomorphism \cite[Section 6]{meyereme:dualities} combined with the Green-Julg isomorphism for proper groupoids \cite[Theorem 4.2]{meyhea:eqrepkth} translate the assembly map $\mu^G_A$ into the map $K_*((P\otimes A)\rtimes G)\to K_*(A\rtimes G)$ induced by $\alpha$. Now it is easy to see from the definition of equivariant Kasparov dual that the element $\tau_A(\alpha)\in\KKK^G(P\otimes A,A)$ is invertible when $A$ is a proper $C^*$-algebra.

\subsection{Main results}

Let us summarize our main results. Recall that $G$ is a countable discrete group with a $G$-compact model for $\E$.

As we have explained in the previous sections, our main strategy for obtaining duality relies on
\begin{enumerate}
\item the $\gamma$ element, or 
\item the $(\gamma)$-element.
\end{enumerate} 
The choice of one over the other does not affect the expression for the unit of Spanier--Whitehead duality, nevertheless the descriptions of the counit and the elements $\Lambda_{C^*_r(G)}$ and $\Lambda_{C_0(\E)\rtimes G}$ depend on the method that we are employing. In practice, the latter elements will be expressible in terms of the $\gamma$-element in the first case, and in the terms of the $(\gamma)$-element in the second case.

Along this categorization, Theorem \ref{thm:thmA} and Corollary \ref{cor:corB} below fall in the first scenario, while Theorem \ref{thm:thmC} is an instance of the second. Section \ref{sec:cons} contains simple examples of possible applications of duality in $K$-theory.

\begin{theoremA}\label{thm:thmA}
Suppose that the $\gamma$-element $\gamma\in \KK{G}{\C}{\C}$ exists and let $P_\C$ be the source of the Dirac morphism $\alpha\in \KK{G}{P_\C}{\C}$. Then the $C^*$-algebra $P_\C\rtimes G$ is Spanier--Whitehead $K$-dual to $C_0(\E)\rtimes G$.
\end{theoremA}

A few more comments about this theorem. The source of the Dirac morphism (the ``simplicial approximation'' in \cite{nestmeyer:loc}) can be obtained in a variety of ways: by appealing to the Brown representability theorem, by considering the left adjoint to the embedding functor of projective objects, or by constructing the appropriate homotopy colimit from a projective resolution of $\C$ (here, ``projective'' is to be understood in a relative sense, i.e., with respect to the homological ideal of weakly contractible objects). Even though $P_\C$ may not be a proper algebra in general, its reduced and maximal crossed products are $\KKK$-equivalent. This is because $P_\C$ belongs to the localizing subcategory of $\KKK^G$ generated by proper algebras and the reduced and maximal crossed product functors are triangulated functors and commute with countable direct sums (see \cite{nestmeyer:loc}).

Theorem A provides a fourth characterization of $P_\C$. Namely as the Spanier--Whitehead $K$-dual of the classifying space for proper actions. Note that even though our statement is only available after descent, that is we can only get $P_\C\rtimes G$ and not $P_\C$ via duality, this is only a minor drawback in the case of discrete groups, for the the left-hand side of \eqref{eq:assmapintro} retains the full information of the ``topological'' $K$-theory group through the dual Green--Julg isomorphism
\[
\KK{G}{C_0(\E)}{\C}\cong \KK{}{C_0(\E)\rtimes G}{\C}.
\]
%so that the source of the Baum--Connes map in the localization approach REF becomes $K_*(P\rtimes G)$.

{In the situation where, at the $\KKK$-theory level, the simplicial approximation is equivalent to the data of $G$ acting on the point, we can replace $P_\C\rtimes G$ with $C^*_r(G)$ and obtain Spanier--Whitehead duality for the group as in the next Corollary. If the $\gamma$-element exists, we define the \emph{strong} Baum--Connes conjecture to be the statement that $\jmath_r^G(\gamma)=1_{C^*_r(G)}$ in $\KK{}{C^*_r(G)}{C^*_r(G)}$.}

\begin{corollaryA}\label{cor:corB}
Suppose the $\gamma$-element exists. Then $G$ has Spanier--Whitehead duality if and only if it satisfies the strong Baum--Connes conjecture.
\end{corollaryA}

{In light of the result above, we can view the notion of Spanier--Whitehead $K$-duality for $G$ as a homotopy-theoretic characterization of the strong Baum--Connes conjecture (cf. Remark \ref{rem:spectra}).}

The main application of the previous corollary is summarized in the result below.

\begin{corollaryA} All a-T-menable groups which admit a $G$-compact model of $\E$ have Spanier--Whitehead $K$-duality.
Examples of a-T-menable groups are the following:
\begin{itemize}
\item All groups which act properly, affine-isometrically and co-compactly on a finite-dimensional Euclidean space.
\item All co-compact lattices of simple Lie groups $\mathrm{SO}(n, 1)$ or $\mathrm{SU}(n,1)$.
\item All groups which act co-compactly on a tree.
\end{itemize}
\end{corollaryA}

Having such an explicit duality should be useful. For example, in principle, it allows us to compute the Lefschetz number of an automorphism of $C^\ast_r(G)$ or more generally of a morphism $f$ in $\KKK(C^\ast_r(G), C^\ast_r(G))$ (see \cite{DEM:traces,eme:lefschetz}).

If a cycle with property $(\gamma)$ is found, then we can deduce the duality in complete analogy with the case of the $\gamma$-element (this is how the definition of property $(\gamma)$ was designed). However in this case we do not have information on the localization at the weakly contractible objects \cite{meyernest:tri}. So we get the corresponding statement for Corollary \ref{cor:corB}, but not for Theorem \ref{thm:thmA}.  

\begin{theoremA}\label{thm:thmC}
Suppose there is a $(\gamma)$-element $x\in \KK{G}{\C}{\C}$ for $G$. If $\jmath^G_r(x)=1_{C^*_r(G)} \in \KK{G}{C^\ast_r(G)}{C^\ast_r(G)}$, then $G$ has Spanier--Whitehead duality.
\end{theoremA}

\section{General framework}

Let $G$ be a countable discrete group, and $\underline{E}G$ be a $G$-compact model of the universal proper $G$-space. Let $A$ and $B$ be $C^*$-algebras equipped with a $G$-action. If the $G$-action on $B$ is trivial, we recall  the dual Green--Julg isomorphism (\cite{black:kth,valjens:misb,markus:ass})
\[
\mathrm{GJ}\colon \KK{G}{A}{B}\cong \KK{}{A\rtimes G}{B}. 
\]

Given $a\in A$, define $\delta_g^a\in C_c(G,A)\subseteq A\rtimes G$ to be the function
\[
\delta_g^a(t)=
\begin{cases}
a&\text{if $t=g$}\\
0&\text{if $t\neq g$.}
\end{cases}
\]%Dirac mass centered at $g\in G$ with value $a$. 
The \emph{dual coaction} is defined as
\begin{align*}
\Delta\colon A\rtimes G&\to C^*_r(G)\otimes A\rtimes G \\
\delta_g^a&\mapsto g\otimes\delta_g^a.
\end{align*}

Let $c\in C_c(\E)$ be a cutoff function, % i.e., $\sum_{g\in G} g(c)^2=1$,%
and consider the associated projection $p_c\in C_c(G,C_0(\E))\subseteq C_0(\E)\rtimes G$ defined by $p_c(g)=cg(c)$. This projection does not depend on $c$ up to homotopy, hence we will denote it $p_G$ in the sequel.
\begin{definition}\label{def:cunit}
We define the \emph{canonical duality unit} to be the class 
\[
\delta=\delta_G=[\Delta(p_G)]\in \KK{}{\C}{C^*_r(G)\otimes C_0(\E)\rtimes G}.
\]
\end{definition}
The notational dependence on $G$ shall be dropped when clear from the context. {In this paper, whenever we say that $G$ has Spanier--Whitehead duality, we implicitly assume that the duality unit is given as above.}

Let us recall the definition of Kasparov's descent homomorphism \cite{kas:descent}, which plays an important role in this paper. It will be denoted $\jmath^G$ below. Suppose $(\phi,H,T)$ is a Kasparov cycle defining an element of $\KK{G}{A}{B}$. The $G$-action on $H$ will be denoted $U\colon G \to \mathrm{End}_\C(H)$. The element $\jmath^G([\phi,H,T])\in \KK{}{A\rtimes G}{B\rtimes G}$ is defined by the cycle $(\tilde{\phi},H\rtimes G, \tilde{T})$ given as follows. 

The Hilbert $C^*$-module $H\rtimes G$ is the completion of $C_c(G,H)$ with respect to the norm induced by the following $B\rtimes G$-valued inner product:
\begin{equation*}\label{eq:inndesc}
\bra{\xi}\ket{\zeta}(t)=\sum_{g\in G}\beta_{g^{-1}}(\bra{\xi(g)}\ket{\zeta(gt)}),
\end{equation*}
where $\xi,\zeta\in C_c(G,H)$, $t\in G$, and $\beta$ denotes the given $G$-action on $B$. The right action of $B\rtimes G$ is uniquely determined by the formula
\[
(\xi \cdot f)(t)=\sum_{g\in G}\xi(g)\beta_g(f(g^{-1}t)),
\]
where $\xi\in C_c(G,H),f\in C_c(G,B)$ and $t \in G$. The representation of $A \rtimes G$ on $H\rtimes G$ is determined by %(using left module notation) as
\begin{equation*}\label{eq:leftdesc}
(\tilde{\phi}(f)(\xi))(t)= \sum_{g\in G}\phi(f(g))[U(g)(\xi(g^{-1}t))],
\end{equation*}
where $f\in C_c(G,A), \xi\in C_c(G,H)$ and $t \in G$. Finally the operator $\tilde{T}$ is defined by $(\tilde{T}\xi)(t)=T(\xi(t))$ for $\xi\in C_c(G,H)$ and $t \in G$. By using reduced crossed products everywhere, we can similarly defined a ``reduced version'' of the descent homomorphism, denoted $\jmath^G_r$ in the sequel.

\begin{lemma}[{\cite[Proposition 4.7]{markus:ass}}]\label{lem:desfact}
Kasparov's descent homomorphism can be factorized as follows:
\begin{equation*}
\xymatrixcolsep{4pc}\xymatrix{\KK{G}{A}{\C} \ar[d]^-{\mathrm{GJ}} \ar[r]^-{\jmath^G} & \KK{}{A\rtimes G}{C^*(G)}\\
\KK{}{A\rtimes G}{\C} \ar[r]^-{\tau_{C^*(G)}} & \KK{}{C^*(G)\otimes A\rtimes G }{C^*(G)} \ar[u]^-{\Delta^*}.}
\end{equation*}
\end{lemma}
When the canonical map $A\rtimes G\to A\rtimes_r G$ is an isomorphism (e.g., if $G$ acts properly on $A$), the version of the previous lemma with \emph{reduced} crossed products also holds. See Remark \ref{r:cross}.

\begin{lemma}[{\cite[Section 2]{valjens:misb}}]\label{lem:gjj}
Let $A$ and $B$ be $G$-$C^*$-algebras and suppose the $G$-action on $B$ is trivial. Consider an element $x\in \KK{G}{A}{A}$. The following diagram commutes.
\begin{equation*}
%\xymatrixcolsep{4pc}
\xymatrix{\KK{G}{A}{B} \ar[r]^-{\mathrm{GJ}} \ar[d]^-{x\hatotimes -} & \KK{}{A\rtimes G}{B} \ar[d]^-{\jmath^G(x)\hatotimes -}\\
\KK{G}{A}{B} \ar[r]^-{\mathrm{GJ}} & \KK{}{A\rtimes G}{B} .}
\end{equation*}
\end{lemma}
%\begin{proof}
%It is easier to prove the statement by considering $\mathrm{GJ}^{-1}=\phi$ in place of $\mathrm{GJ}$. Given a cycle $(H,\pi,T)$ for $\KK{}{A\rtimes G}{B}$, the map $\phi$ equips $H$ with a $G$-action by considering an approximate unit for $A$ and the representation $\pi$. Now, the representation in $\phi(\jmath^G(x)\hatotimes [H,\pi,T])$ is no longer non-degenerate, and its non-degenerate closure is easily seen to be isomorphic to $x\hatotimes \phi([H,\pi,T])$.
%\end{proof}

It follows from Lemma \ref{lem:desfact} that we have the following commutative diagram 
\begin{align*}
\xymatrixcolsep{7pc} \xymatrix{
\KKK^G(C_0(\E), B) \ar[d]_-{\mathrm{GJ}}^{\cong}  \ar[r]^-{\mu^G_B} & \KKK(\mathbb{C}, B\otimes C^\ast_r(G)) \ar[d]^-{=}    \\
\KKK(C_0(\E)\rtimes G, B)   \ar[r]^-{\delta\,\hot_{C_0(\E)\rtimes G}-}   & \KKK(\mathbb{C}, B\otimes C^\ast_r(G))  
}
\end{align*} 

Since the definition of the duality counit requires additional information, and will depend on the choice of ``$\gamma$-like'' element, the rest of this section gets split in two parts. The torsion-free case is treated in detail in Subsection \ref{subsec:tfree}
%\begin{equation}\label{eq:assmap}
%\xymatrixcolsep{7pc}\xymatrix{\KK{}{C_0(\E)\rtimes G}{\C}  \ar[r]_-{[\Delta(p)]\hot_{C_0(\E)\rtimes G}-}  & \KK{}{\C}{C^*_r(G)}.}
%\end{equation}
\subsection{Argument based on the \texorpdfstring{$(\gamma)$}{(gamma)}-element}

Let $(H, T)$ be a $G$-equivariant Kasparov cycle with property $(\gamma)$. Let $x=[H,T]$ be the corresponding element in $\KKK^G(\mathbb{C}, \mathbb{C})$.  Let 
\begin{equation}\label{eq:gammamorphism}
\tilde{x}=[H\otimes \ell^2(G), \rho\otimes\pi, (g(T))_{g\in G}] \in \KK[]{G}{C^*_r(G)\otimes C_0(\E)}{\C}.
\end{equation}
Here, $\pi\colon C_0(\E)\to B(H)$ is the representation witnessing the conditions for property $(\gamma)$ of $(H,T)$, $\rho$ stands for the right regular representation, and $C^*_r(G)$ has trivial $G$-action. By means of the Green-Julg isomorphism, we set
\[ 
d=\mathrm{GJ}(\tilde{x}) \in \KKK(C^\ast_r(G)\otimes C_0(\E)\rtimes G , \mathbb{C}).
\]
%The pullback $\sigma^*$ will be suppressed from notation for ease of reading.
%\begin{proposition}
%We have the equality $(\delta_j\circ d_j)(y)=y\,\hot_{C^*_r(G)}\,\jmath_r^G(x)$. In particular, if $x=1\in \KK[0]{G}{\C}{\C}$ then $C^*_r(G)$ is Poincar\'e dual to $C_0(\E)\rtimes G$.
%\end{proposition}
%This follows from \cite[Proposition 4.2]{nish:dsplitb}.
We set $\Lambda_{C^\ast_r(G)}=\delta\,\hot_{C_0(\E)\rtimes G}\, d$ and  $\Lambda_{C_0(\E)\rtimes G}=\delta\,\hot_{C^\ast_r(G)}\, d$. We shall prove
\begin{enumerate}
\item $\Lambda_{C^\ast_r(G)}=\jmath^G_r(x)$ in $\KKK(C^\ast_r(G), C^\ast_r(G))$;
\item $\Lambda_{C_0(\E)\rtimes G}=1_{C_0(\E)\rtimes G}$ in $\KKK(C_0(\E)\rtimes G, C_0(\E)\rtimes G)$.
\end{enumerate}

\begin{proposition}\label{prop:gammaG}
We have the equality $\Lambda_{C^*_r(G)}=\jmath_r^G(x)$.
\end{proposition}
\begin{proof} 
We claim the Kasparov module
\begin{equation}\label{eq:kasiso}
[p_G]\,\hot_{C_0(\E)\rtimes G}\,\jmath^G_r(\tilde{x})
\end{equation}
is equivalent to $\jmath_r^G(x)$, i.e., there is an isomorphism of Hilbert $C^*$-modules intertwining the representations and the operators (up to a compact perturbation). 

The class in \eqref{eq:kasiso} is represented by
\[
(H\otimes \ell^2(G)\rtimes_r G, (\rho\otimes\pi\rtimes_r 1)(p_G\otimes -), (g(T))_{g\in G}\rtimes_r 1).
\]
We have an isomorphism of $C^*_r(G)$-modules
\begin{equation}\label{eq:isoiso}
H\rtimes_r G\cong (\rho\otimes\pi\rtimes_r 1)(p_G\otimes 1)(H\otimes \ell^2(G)\rtimes_r G)
\end{equation}
given by the assignment
\[
\xi\rtimes_r u_g \mapsto \sum_{h\in G}\pi(c)(h\cdot \xi)\otimes \delta_h\rtimes_r u_{hg},
\]
where $\xi\in H$, $\delta_h\in \ell^2(G)$, and $c$ is a cutoff function defining $p_G$. The inverse of the map above is given by the restriction of
\[
(\xi)_{h\in G}\rtimes_r u_g \mapsto \sum_{h\in G}h^{-1}\cdot (\pi(c)\xi_h)\otimes \rtimes_r u_{h^{-1}g},
\]
where $(\xi)_{h\in G}\in H\otimes\ell^2(G)$. Under the isomorphism in \eqref{eq:isoiso}, the representation $(\rho\otimes\pi\rtimes_r 1)(p_G\otimes -)$ is identified with the left action of $C^*_r(G)$ on $H\rtimes_r G$, and the compressed operator
\[
(\rho\otimes\pi\rtimes_r 1)(p_G\otimes 1)((g(T))_{g\in G}\rtimes_r 1)(\rho\otimes\pi\rtimes_r 1)(p_G\otimes 1)
\]
is identified with $T^\prime\rtimes_r 1$ on $H\rtimes_r G$, where we defined
\[
T^\prime=\sum_{g\in G} (g\cdot c)T(g\cdot c).
\]
Hence the claim follows by definition of property $(\gamma)$.

By Lemma \ref{lem:desfact}, we have
\[
\jmath^G_r(\tilde{x}) =  \Delta\otimes_{C_0(\E)\rtimes G}\mathrm{GJ}(\tilde x).
\]
%Applying Lemma \ref{lem:desfact}, using morphism notation, we see that
%\begin{equation*}
%\jmath^G_r(\tilde{x})=[\Delta^\prime] \otimes_{C_0(\E)\rtimes G \otimes C^*_r(G) \otimes C^*_r(G)} (\mathrm{GJ}(\tilde x)\otimes 1_{C^*_r(G)}) ,
%\end{equation*}
%where the dual coaction $\Delta^\prime$ is such that
%\[
%[\Delta^\prime]\in \KK[0]{}{ C_0(\E)\rtimes G\otimes C^*_r(G)}{C_0(\E)\rtimes G\otimes C^*_r(G)\otimes 
%C^*_r(G)}. 
%\]
%In other words, $[\Delta^\prime]=[\Delta]\hatotimes 1_{C^*_r(G)}$. Overall we have
Thus, we have
\begin{align*}
\jmath^G_r(x)&=[p_G]\,\hot_{C_0(\E)\rtimes G}\,\jmath^G_r(\tilde{x})\\
&=[p_G]\hatotimes_{C_0(\E)\rtimes G}  \Delta\otimes_{C_0(\E)\rtimes G}\mathrm{GJ}(\tilde x)\\
&=\delta\,\hot_{C_0(\E)\rtimes G}\, d.
\end{align*}

%\begin{align*}
%\jmath^G_r(x)&=[p]\,\hot_{C_0(\E)\rtimes G}\,\jmath^G_r(\tilde{x})\\
%&=[p]\hatotimes (([\Delta]\hatotimes 1_{C^*_r(G)}) \hatotimes 1_{C^*_r(G)} \hatotimes \mathrm{GJ}(\tilde{x}))\\
%&=( ([p]\hatotimes [\Delta])\hatotimes 1_{C^*_r(G)}) \hatotimes (1_{C^*_r(G)} \hatotimes \mathrm{GJ}(\tilde{x}))\\
%&=\delta\,\hot_{C_0(\E)\rtimes G}\, d.
%\end{align*}
\end{proof} 

\begin{proposition}\label{prop:czeroone}
We have the equality $\Lambda_{C_0(\E)\rtimes G}=1_{C_0(\E)\rtimes G}$.
\end{proposition}

In order to prove the proposition, a few preliminaries are in order. First we generalize the construction in \eqref{eq:gammamorphism} to include a coefficient algebra. This is easily done: simply replace $\ell^2(G)$ with the right Hilbert $A$-module $\ell^2(G,A)$ and define the right regular representation $\rho^G_A$ of $A\rtimes_r G$ (equipped with trivial $G$-action)
\[
a \mapsto (g(a))_{g\in G},\qquad  h \mapsto \rho_h : (a_g)_{g\in G} \mapsto (a_{gh})_{g\in G}
\] 
for $a \in A, h \in G$. Thus we get a class $\tilde{x}_A$ in $\KK[]{G}{A\rtimes_r G\otimes C_0(\E)}{A}$.
We define a group homomorphism
\begin{equation*}
\nu^G_A\colon \KK[]{}{\C}{ A\rtimes_r G} \to  \KK[]{G}{C_0(\E)}{A}
\end{equation*}
as the one induced by the class $\tilde{x}_A$ via the index pairing
\[
\KKK(\C,A \rtimes_r G) \times \KKK^G(A\rtimes_r G\otimes C_0(\E), A) \to \KKK^G(C_0(\E),A).
\]
This map is referred to as the $(\gamma)$-morphism in \cite{nish:dsplitb}. Note also that $\mathrm{GJ}\circ\nu^G_\C$ equals the map $d_j$ from Definition \ref{def:wpoin} (choosing $B=C_0(\E)\rtimes G$ as usual). The lemma below is about the naturality property of the assembly map and the $(\gamma)$-morphism.

\begin{lemma}
The following diagrams commute for any $f\in \KK[]{G}{A}{B}$.
\begin{align*}
&\xymatrix{\KK[]{G}{C_0(\E)}{A} \ar[r]^-{\mu^G_A} \ar[d]^-{-\hatotimes f } & \KK[]{}{\C}{A\rtimes_r G} \ar[d]^-{-\hatotimes \jmath^G_r (f)}\\
\KK[]{G}{C_0(\E)}{B}\ar[r]^-{\mu^G_A} & \KK[]{}{\C}{B\rtimes_r G},}
\\[6pt]
&\xymatrix{\KK[]{}{\C}{A\rtimes_r G} \ar[r]^-{\nu^G_A} \ar[d]^-{-\hatotimes \jmath^G_r (f) } &  \KK[]{G}{C_0(\E)}{A}  \ar[d]^-{-\hatotimes f}\\
\KK[]{}{\C}{B\rtimes_r G} \ar[r]^-{\nu^G_A} & \KK[]{G}{C_0(\E)}{B}
}
\end{align*}
\end{lemma}
\begin{proof}
The first diagram commutes by functoriality of descent and associativity of the Kasparov product. By results from \cite{meyer:genhom} any morphism $f$ in $\KKK^G(A, B)$ can be written as a composition of $*$-homomorphisms and their inverses in $\KKK$. This means it suffices to check the commutativity of the second diagram with respect to $*$-homomorphisms $f\colon A\to B$. We omit this simple verification.
\end{proof}

\begin{proof}[Proof of Proposition \ref{prop:czeroone}] Let $B=C_0(\E)\rtimes G$ and regard it as a $G$-$C^\ast$-algebra with the trivial $G$-action. We have the following diagram
\begin{align*}
\xymatrix{
\KKK^G(C_0(\E), B) \ar[d]_-{\mathrm{GJ}}^{\cong}  \ar[r]^-{\mu^G_B} & \KKK(\mathbb{C}, C^\ast_r(G)\otimes B) \ar[d]^-{=} \ar[r]^-{\nu^G_B} &\KKK^G(C_0(\E), B) \ar[d]_-{\mathrm{GJ}}^{\cong}     \\
\KKK(B, B)   \ar[r]^-{\delta\,\hot_B\,-}    & \KKK(\mathbb{C}, C^\ast_r(G)\otimes B)   \ar[r]^-{-\,\hot_{C^\ast_r(G)}\, d}   &\KKK(B, B)  
}
\end{align*} 
If we prove the composition on the top is the identity, then it follows $\Lambda_{C_0(\E)\rtimes G}=1_{C_0(\E)\rtimes G}$. Let $D_B\colon P_B\to B$ be a weak equivalence as in \cite{nestmeyer:loc}. Because the following diagram commutes,
\begin{align*}
\xymatrix{
\KKK^G(C_0(\E), P_B) \ar[d]^-{D_{B*}}  \ar[r]^-{\mu^G_{P_B}} & \KKK(\mathbb{C}, P_B\rtimes G) \ar[d]^-{\jmath^G_r(D_{B*})} \ar[r]^-{\nu^G_{P_B}} &\KKK^G(C_0(\E), P_B) \ar[d]^-{D_{B*}}     \\
\KKK^G(C_0(\E), B) \ar[r]^-{\mu^G_B} & \KKK(\mathbb{C}, B\rtimes_r G)  \ar[r]^-{\nu^G_B} &\KKK^G(C_0(\E), B),
}
\end{align*} 
it suffices to show that $\nu^G_{P_B}$ is a left inverse of the assembly map $\mu^G_{P_B}$. Now, $\mu^G_{P_B}$ is invertible, hence it suffices to show that $\nu^G_{P_B}$ yields a right inverse. A minor generalization of the proof of Proposition \ref{prop:gammaG} shows that $\mu^G_{P_B}\circ \nu^G_{P_B}$ coincides with the induced action of $x\in\KKK^G(\mathbb{C}, \mathbb{C})$ on $K_*(P_B\rtimes G)$. Recall that $x$ equals the identity when restricted to each finite subgroup $H\subseteq G$, and $P_B\rtimes G$ belongs to the localizing subcategory of $\KKK$ generated by the $B\rtimes H$'s. Therefore the map $x\hatotimes - \colon K_*(P_B\rtimes G)\to  K_*(P_B\rtimes G)$ is the identity by \cite[Theorem 9.3]{nestmeyer:loc}.
\end{proof}

\begin{remark}
In parallel with Proposition \ref{prop:gammaG}, one can prove that 
\begin{equation*}
\Lambda_{C_0(\E)\rtimes G}=\jmath^G_r (x\hatotimes 1_{C_0(\E)}).
\end{equation*}
Again, we set $B=C_0(\E)\rtimes G$ and first notice that $ \nu^G_B\circ \mu^G_B=x\,\hot_{C_0(\E)}\,-$. It is enough to show this when $B$ is replaced by $P_B$, in which case we can invert the assembly map and write
\begin{align*}
(\jmath^G_r(x)\hatotimes -) &=(\jmath^G_r(x)\hatotimes -)\circ \mu^G_B \circ (\mu^G_B)^{-1}\\
\mu^G_B \circ \nu^G_B &=\mu^G_B\circ (x\,\hot_{C_0(\E)}\,-)\circ (\mu^G_B)^{-1}\\
\nu^G_B\circ \mu^G_B&=x\,\hot_{C_0(\E)}\,-.
\end{align*}
To complete the proof one must show that
\[
\mathrm{GJ}(x\,\hot_\C\, \mathrm{GJ}^{-1}(1_B))=\jmath^G_r(x\hatotimes 1_{C_0(\E)})\,\hot_B\, 1_B,
\]
but this follows from Lemma \ref{lem:gjj}.
\end{remark}

We now come to the main result of this subsection.

\begin{theorem}
Suppose there is a $(\gamma)$-element $x\in \KK{G}{\C}{\C}$ for $G$. If $\jmath^G_r(x)=1 \in \KK{G}{C^\ast_r(G)}{C^\ast_r(G)}$, then $G$ has Spanier--Whitehead duality.
\end{theorem}

%Let $d$ be the class in $\KKK(C_0(EG)\rtimes_rG\otimes C^\ast_r(G), \mathbb{C})$ as before. The $C^\ast$-algebras $C_0(\E)\rtimes G$ and $C(BG)$ are Morita-equivalent. We set
%\[
%\underline d \in \KKK(C(BG)\otimes C^\ast_r(G), \mathbb{C})
%\]
%to be the element which corresponds to $d$. We set 
%\[
%\underline\Lambda_{C^\ast_r(G)}=[\mathrm{MF}]\hot_{C(BG)} \underline d, \underline\Lambda_{C(BG)}=[\mathrm{MF}]\hot_{C^\ast_r(G)} \underline d.
%\]
%The following is immediate from the previous arguments.
%\begin{theorem} We have $\underline\Lambda_{C^\ast_r(G)}=\jmath^G_r(x)$ and $\underline\Lambda_{C(BG)}=1_{C(BG)}$.
%\end{theorem}

\subsection{Argument based on the \texorpdfstring{$\gamma$}{gamma}-element}

Suppose there is a gamma element $\gamma$ as in Definition \ref{def:gamma}. Following \cite[Chapter 15]{higguetrout:eqeth}, we define a map $s_A$ for any proper algebra $A$. This is the $G$-equivariant $\ast$-homomorphism 
\[
s_A\colon A\rtimes_rG\otimes C_0(\E) \to K(A\otimes \ell^2(G))
\]
where $A\rtimes_rG$ is equipped with the trivial $G$-action, defined as the tensor product of the following representation of $C_0(\E)$ on $A\otimes \ell^2(G)$:
\[
C_0(\E) \ni\phi \mapsto (\phi)_{g\in G} \in L(A\otimes \ell^2(G))
\]
and the right regular representation of $A\rtimes_rG$ on $A\otimes \ell^2(G)$:
\[
A \ni a \mapsto (g(a))_{g\in G} \in L(A\otimes \ell^2(G)), \,\,\, G \ni g \mapsto 1\otimes\rho_g
\]
where $\rho_g$ is the right translation by $g$. Here, the $G$-action on the Hilbert module $A\otimes \ell^2(G)$ is given by the tensor product of the action on $A$ and the left-regular representation. The $\ast$-homomorphism $s_A$ defines an element
\[
s_A \in \KKK^G(A\rtimes_rG\otimes C_0(\E), A).
\]
\begin{proposition}[{cf. \cite[Chapter 15]{higguetrout:eqeth}}] \label{prop_proper_inverse} 
For any proper $G$-$\Calg$ $A$, the $\ast$-homomorphism $s_A$ defines the inverse
\[
s_A\colon \KKK(\mathbb{C}, A\rtimes_rG) \to \KKK^G(C_0(\E), A)
\] 
of the assembly map 
\[
\mu^G_A\colon \KKK^G(C_0(\E), A) \to \KKK(\mathbb{C}, A\rtimes_rG).
\]
\end{proposition}
\begin{proof} 
The assembly map is an isomorphism for any proper algebra. Hence, we just show that the composition $\mu^G_A\circ s_A$ is the identity. Take a Kasparov cycle $(E, F)$ for $\KKK(\C, A\rtimes_rG)$ where $E$ is a graded $A\rtimes_rG$-module and $F$ is an odd, self-adjoint operator on $E$ satisfying $1-F^2 \equiv 0$ modulo compact operators. 

By Kasparov's stabilization theorem, we can assume $E$ is $A\otimes H\rtimes_rG$ for some graded Hilbert space $H$ with the trivial $G$-action. The map $s_A$ sends this cycle $(A\otimes H\rtimes_rG, F)$ to the $G$-equivariant cycle $(A\otimes H \otimes \ell^2(G), \pi, \rho(F))$ for $\KKK^G(C_0(\E), A)$ where $\pi$ is a representation of $C_0(\E)$ on $A\otimes H \otimes \ell^2(G)$ defined as: for $\phi$ in $C_0(\E)$
\[
\pi(\phi)(a_g\otimes v_g \otimes \delta_g) =\phi a_g\otimes v_g \otimes \delta_g
\]
and $\rho(F)$ is an operator in $L(A\otimes H \otimes \ell^2(G))$ determined by the map
\[
L(A\otimes H\rtimes_rG)=M(A\otimes K(H)\rtimes_rG) \xrightarrow{\rho} M(A\otimes K(H\otimes \ell^2(G))) = L(A\otimes H \otimes \ell^2(G)),
\]
which is a natural extension of the right regular representation $\rho^G_A$ of $A\rtimes_rG$ on $A\otimes \ell^2(G)$ described before. Hence, the composition $\mu^G_A\circ s_A$ sends the cycle $(A\otimes H\rtimes_rG, F)$ to the cycle $(p_c(A\otimes H \otimes \ell^2(G)\rtimes_rG), p_c\rho(F)\rtimes_r1p_c)$ where we simply denote by $p_c$ the image of a cutoff projection $p_c$ in $C_0(\E)\rtimes_ G$ by the representation $\pi\rtimes_r1$. 

On the other hand, there is an isomorphism of right Hilbert $A\rtimes_r G$-modules
\[
A\otimes H \rtimes_r G \longrightarrow  p_c (A\otimes H \otimes \ell^2(G)\rtimes_r G)
\]
given by for $\xi$ in $A\otimes H$,
\[
\xi\rtimes_r u_g  \mapsto  \sum_{h\in G} c(h(\xi))\otimes\delta_h\rtimes_r u_{hg}.
\]
The inverse map is given by for $(\xi_h)_{h\in G}$ in $A\otimes H\otimes \ell^2(G)$,
\[
(\xi_h)_{h\in G}\rtimes_r u_g \mapsto \sum_{h\in G} h^{-1}(c \xi_h)\rtimes_r u_{h^{-1}g}.
\]
Under this isomorphism, the restriction $p_c\rho(F)\rtimes_r1p_c$ of $\rho(F)\rtimes_r1$ on the subspace $p_c (A\otimes H \otimes \ell^2(G)\rtimes_r G)$ of $A\otimes H \otimes \ell^2(G)\rtimes_r G$ corresponds to the operator $F$ on $A\otimes H \rtimes_r G$. In summary, the composition $\mu^G_A\circ s_A$ sends the cycle $(A\otimes H\rtimes_rG, F)$ to itself up to the isomorphism described above.

\end{proof}
For any separable $G$-$\Calg$ $B$, we have the following commutative diagram 
\begin{align*}
\xymatrix{
\mu^G_B\colon  \KKK^G(C_0(EG), B) \ar[d]_-{}^{-\,\hot_{\mathbb{C}}\beta}   \ar[r]& \KKK(\mathbb{C}, B\rtimes_rG) \ar[d]^-{-\,\hot_{B\rtimes_rG}\jmath^G_r(\mathrm{id}_B\,\hot\, \beta)}   \\
\mu^G_{B\otimes P} \colon   \KKK^G(C_0(EG), B\otimes P) \ar[d]_-{}^{-\,\hot_{P}\alpha }  \ar[r]^-{\cong}& \KKK(\mathbb{C}, (B\otimes P)\rtimes_rG) \ar[d]^-{-\,\hot_{(B\otimes P)\rtimes_rG}\jmath^G_r(\mathrm{id}_B\,\hot\, \alpha)}   \\
\mu^G_B\colon    \KKK^G(C_0(EG), B)   \ar[r]& \KKK(\mathbb{C}, B\rtimes_rG)
}
\end{align*} 
where the vertical composition on the left is the identity. With this and by Proposition \ref{prop_proper_inverse}, we see that the element
\[
  (\jmath^G_r(1_B\,\hot\, \beta))\,\hot_{(B\hatotimes P)\rtimes_rG}\, s_{B\otimes P}\,\hot_{P}\, \alpha \in \KKK^G((B\rtimes_rG)\otimes C_0(\E), B)
 \]
 induces the left-inverse of the assembly map $\mu^G_B$ via Kasparov product. We remark that this is the standard technique for proving the split injectivity of the assembly map in the presence of a $\gamma$-element. 

Now, we set $d'$ to be the element in $\KKK( C^\ast_r(G)\otimes C_0(\E)\rtimes G, \mathbb{C})$ which corresponds to 
\[
 d''=(\jmath^G_r(\beta))\,\hot_{P\rtimes_rG} s_P\,\hot_{P} \alpha \in \KKK^G(C^\ast_r(G)\otimes C_0(\E), \mathbb{C}).
\]
Let
\[
\delta=\delta_G \in \KKK(\mathbb{C}, C_0(\E)\rtimes G\otimes C^\ast_r(G))
\]
as before. We set $\Lambda'_{C^\ast_r(G)}=\delta\,\hot_{C_0(\E)\rtimes G}\, d', \Lambda'_{C_0(\E)\rtimes G}=\delta\,\hot_{C^\ast_r(G)} \,d'$. 

\begin{proposition}\label{prop:main}
We have 
\[
\Lambda'_{C^\ast_r(G)}  = \jmath^G_r(\gamma) \in \KKK(C^\ast_r(G), C^\ast_r(G)).
\]
and
\[
  \Lambda'_{C_0(\E)\rtimes G}  = 1_{C_0(\E)\rtimes G} \in \KKK(C_0(\E)\rtimes G, C_0(\E)\rtimes G)
\]
\end{proposition}

Before giving a proof of Proposition \ref{prop:main}, let us obtain our main results as its direct consequences:

\begin{theorem}\label{thm:swdm}
If $\jmath^G_r(\gamma)=1_{C^\ast_r(G)}$, then $G$ has Spanier--Whitehead duality. 
\end{theorem}

The previous result has a converse, see Theorem \ref{thm:equivswdbc} for further details.

\begin{theorem} 
If $\mu^G_\C$ is an isomorphism, $G$ has weak Spanier--Whitehead duality. 
\end{theorem}

\begin{theorem}\label{thm:pdual}
In general, if $\gamma\in\KKK^G(\C,\C)$ exists, then $C_0(\E)\rtimes G$ is a Spanier--Whitehead $K$-dual of $P_{\C}\rtimes G$. 
\end{theorem}
\begin{proof} Note that $P_{\C}\rtimes G$ is a direct summand (in the categoy $\KKK$) of $C^*_r(G)$ corresponding to the idempotent $\jmath^G_r(\gamma) \in \KKK(C^\ast_r(G), C^\ast_r(G))$ (see \cite[Proposition 1.6.8]{nee:tri})). Namely, we have
\[
i_{P_\C\rtimes G}\in \KKK(P_\C\rtimes G, C^*_r(G)), \,\,\, q_{P_\C\rtimes G}\in \KKK(C^*_r(G), P_\C\rtimes G), 
\]
so that $q_{P_\C\rtimes G}\circ i_{P_\C\rtimes G}=1_{P_\C\rtimes G}$ and $i_{P_\C\rtimes G}\circ q_{P_\C\rtimes G}=\jmath^G_r(\gamma)$. We set
\begin{align*}
d_{P_\C\rtimes G} = i_{P_\C\rtimes G}\,\hot_{C^*_r(G)} \,d' \in &\KKK(C_0(\E)\rtimes G\otimes P_\C\rtimes G, \C),\\
\delta_{P_\C\rtimes G} = \delta\,\hot_{C^*_r(G)}\, q_{P_\C\rtimes G} \in &\KKK(\mathbb{C}, C_0(\E)\rtimes G\otimes P_\C\rtimes G).
\end{align*}
Then, we have 
\[
\delta_{P_\C\rtimes G} \,\hot_{C_0(\E)\rtimes G}\, d_{P_\C\rtimes G}=1_{P_\C\rtimes G}, \,\,\, \delta_{P_\C\rtimes G} \,\hot_{P_\C\rtimes G}\, d_{P_\C\rtimes G}=1_{C_0(\E)\rtimes G}.
\]
 This proves the statement. We only prove the first identity, the other one is similarly proved. For any $C^*$-algebra $D$, we have the following commutative diagram:
 \begin{align*}
\xymatrix{
 \KKK(P_\C\rtimes G, D) \ar[d]_-{}_{\delta_{P_\C\rtimes G}\,\hot_{P_\C\rtimes G}\,-}   \ar[r] & \KKK(C^*_r(G), D) \ar[d]_{\delta\,\hot_{C^*_r(G)}\,-}   \\
 \KKK(\C, C_0(\E)\rtimes G\otimes D) \ar[d]_-{}^{-\,\hot_{C_0(\E)\rtimes G}\,d_{P_\C\rtimes G} }  \ar[r]^-{=}& \KKK(\mathbb{C}, C_0(\E)\rtimes_rG\otimes D) \ar[d]^{-\,\hot_{C_0(\E)\rtimes_rG}\,d' }   \\
 \KKK(P_\C\rtimes G, D)   \ar[r] & \KKK(C^*_r(G), D).
}
\end{align*} 
Here, the top and the bottom horizontal arrows are induced by $i_{P_\C\rtimes G}$ and $q_{P_\C\rtimes G}$. The right vertical composition is induced by $\jmath^G_r(\gamma)$. It follows, the left vertical composition is the identity. Taking $D=P_\C\rtimes G$, we get 
\[
\delta_{P_\C\rtimes G} \,\hot_{C_0(\E)\rtimes G}\, d_{P_\C\rtimes G}=1_{P_\C\rtimes G}.
\]
\end{proof}
 
\begin{proof}[Proof of Proposition \ref{prop:main}]
We directly compute and prove
\[
\delta\,\hot_{C_0(\E)\rtimes G}\,d' = \jmath^G_r(\gamma) \in \KKK(C^\ast_r(G), C^\ast_r(G)).
\]
For simplicity, we prove this for the case when $\beta$ is represented by a cycle $(P, b)$ where $b$ is an essential unitary in $M(P)$ and if $\alpha$ is represented by a cycle $(H, F)$ where $P$ is represented on $H$ non-degenerately and $F$ is a $G$-equivariant essential unitary modulo $P$. Then, $d''$ is represented by the cycle of the form
\[
(H\otimes \ell^2(G), \rho\otimes\pi, N(g(b))_{g\in G} + M(g(F))_{g\in G})
\]
where the $G$-action on $H\otimes \ell^2(G)$ is the tensor product of the $G$-action on $H$ and the left regular representation on $\ell^2(G)$, $\pi$ is a representation of $C_0(\E)$ on $H\otimes \ell^2(G)$ given by $\phi\mapsto (\phi)_{g\in G}$ and where $\rho$ is a representation of $C^\ast_r(G)$ on $H\otimes \ell^2(G)$ by the right regular representation $g\mapsto 1\otimes\rho_g$. Here, $M$ and $N$ are given by Kasparov Technical Theorem as usual \cite{hig:ktecht,kas:consp,kas:opKfunct}. If we compute $\delta\otimes_{C_0(\E)\rtimes G}d'$, we get the cycle isomorphic to 
\[
(H\rtimes_rG, \pi_G, T_0\rtimes_r1) = \jmath^G_r((H, T_0))
\]
where $(H, T_0)$ is a cycle for $\KKK^G(\mathbb{C}, \mathbb{C})$, $\pi_G$ is the natural left multiplication by $C^\ast_r(G)$ and $T_0= N_0b + M_0F_0$. Here $F_0$ is the average of $F$: $F_0=\int_G g(c)Fg(c)\,d\mu_G$ and so are $N_0$ and $M_0$. The cycle $(H, T_0)$ is (homotopic to) a Kasparov product of $\alpha$ and $\beta$. In other words, the element $[H, T_0]$ is the gamma element $\gamma$. It follows 
\[
\delta\,\hot_{C_0(\E)\rtimes G}\,d' = \jmath^G_r(\gamma).
\]
Now, we can prove
\[
\delta\,\hot_{C^\ast_r(G)}\,d' = 1_{C_0(\E)\rtimes G} \in \KKK(C_0(\E)\rtimes G, C_0(\E)\rtimes G)
\]
using a simple trick. We have the following diagram for $B=C_0(\E)\rtimes G$ with the trivial $G$ action:

\begin{align*}
\xymatrixcolsep{3.8pc}\xymatrix{
\KKK^G(C_0(\E), B) \ar[d]^{\cong}  \ar[r]^-{\mu^G_B} & \KKK(\mathbb{C}, C^\ast_r(G)\otimes B) \ar[d]^{=} \ar[r]^-{(\mu^G_B)^{-1}} &  \KKK^G(C_0(\E), B) \ar[d]^{\cong} \\
\KKK(B, B)   \ar[r]^-{\delta\,\hot_{B}\,-}  & \KKK(\mathbb{C}, C^\ast_r(G)\otimes B)  \ar[r]^-{-\,\hot_{C^*_r(G)}\,d}  & \KKK(B, B).
}
\end{align*} 
Here, by $(\mu^G_B)^{-1}$ we simply mean the left inverse of $\mu^G_B$. This shows $\delta\otimes_{C^\ast_r(G)}d'$ acts as the identity on $\KKK(B, B)$, proving the claim.
\end{proof}

\begin{remark}
The previous proof also shows that $d=d^\prime$, as it is intuitive from the fact that the $\gamma$-element can be represented by a cycle satisfying property $(\gamma)$ \cite{nish:dsplitb}.
\end{remark}

\begin{remark}\label{rem_Kasparovdual}
It is natural to use the duality class $\Theta$ from Subsection \ref{subsec:ekd} to prove Theorem \ref{thm:pdual}. The argument is based on the following diagram, where we have set $d^\prime=\mathrm{GJ}(s_P\,\hot_P\,\alpha)$, and $\mu^G_{P\rtimes G,P}$ is a bivariant assembly map (cf. Section \ref{sec:cons}).
\begin{equation*}
\xymatrixcolsep{5pc}\xymatrix{%
 \KK{}{\C}{P\rtimes G\otimes C_0(\E)\rtimes G} \ar[r]^-{-\,\hot_{C_0(\E)\rtimes G}\,d^\prime } & \KK{}{P\rtimes G}{P\rtimes G}\\
\KK{G}{C_0(\E)}{P\otimes C_0(\E)\rtimes G} \ar[d]^-{p^*_{\E}}_{\cong} \ar[u]^-{\mu^G_{P\otimes C_0(\E)\rtimes G}}_{\cong} \ar[r]^-{-\,\hot_{C_0(\E)\rtimes G}\,d^\prime } & \KK{G}{C_0(\E)\otimes P\rtimes G }{P} \ar[u]^-{\mu^G_{P\rtimes G,P}} \ar[d]^-{p^*_{\E}}_{\cong}\\
\RKK{G}{C_0(\E)}{P\otimes C_0(\E)\rtimes G}  \ar[r]^-{-\,\hot_{C_0(\E)\rtimes G}\,d^\prime } & \RKK{G}{C_0(\E)\otimes P\rtimes G }{P}
}
\end{equation*}
Set $e=\mathrm{GJ}^{-1}(1_{C_0(\E)\rtimes G})$ and consider the element $\delta_0=\Theta\,\hot_{C_0(\E)}\, e$ in the bottom left group. Suppressing $p^*_{\E}$ from the notation, we compute
\begin{align*}
\delta_0\,\hot_{C_0(\E)\rtimes G}\, d^\prime &=\Theta\,\hot_{C_0(\E)}\, e\,\hot_{C_0(\E)\rtimes G}\, \mathrm{GJ}(s_P\,\hot_P\,\alpha) \\
&=( \Theta\,\hot_P\,\alpha)\,\hot_\C (e\,\hot_{C_0(\E)\rtimes G}\, \mathrm{GJ}(s_P))=s_P.
\end{align*}
Now it is routine to check that $\mu^G_{P\rtimes G,P}(s_P)=1_{P\rtimes G}$. Hence if we define
\[
\delta_{P\rtimes G}\in \KK{}{\C}{P\rtimes G\otimes C_0(\E)\rtimes G}
\]
by sending $\delta_0$ through the left vertical isomorphism in the diagram above,  we have 
\[
\delta_{P\rtimes G}\,\hot_{C_0(\E)\rtimes G}\, d^\prime=1_{P\rtimes G}.
\] 
The other identity is similarly proved, we skip the details. 

Note that this is an improvement over Theorem \ref{thm:pdual}, because the existence of $\Theta$ is strictly weaker than having a gamma element. A similar argument shows that in general, if $P_\C$ is a (categorical) direct summand of some proper algebra, the conclusion of Theorem \ref{thm:pdual} holds, namely $C_0(\E)\rtimes G$ is a Spanier--Whitehead $K$-dual of $P_{\C}\rtimes G$. 
\end{remark}

\subsection{The torsion-free case}\label{subsec:tfree}
We treat the torsion-free case separately, partly because it is particularly simple (e.g., condition (1) of Definition \ref{def:gamma} reduces to a statement in non-equivariant $K$-theory), partly because it is among the first cases where the duality classes (i.e., unit and counit) have been identified (albeit in a slightly different language, cf. \cite[Theorems 6.6 and 6.7]{kas:descent}). 

We assume that $G$ is a countable, discrete, torsion-free group. In this case, because proper actions are automatically free, the space $\underline{E}G$ is identified as the total space $EG$ of the classifying space for principal $G$-bundles, and our assumption that $G$ admits a $G$-compact model of $\E$ translates into the assumption that $G$ admits a compact model of $BG$. Denote by $[\mathrm{MF}]$, the class
\[
[\mathrm{MF}] \in \KKK(\mathbb{C}, C^\ast_r(G)\otimes C(BG))
\]
associated to the module of sections of the Mi\v{s}\v{c}enko bundle. This is the Hermitian bundle of $C^*$-algebras obtained from the associated bundle construction 
\[
EG \times_G C^*_r(G)\to BG,
\]
where $G$ acts diagonally, acting on the reduced group $C^*$-algebra via the left regular representation \cite{mf:bundle}.

\begin{proposition}[\cite{connes:ncg}, for a proof see \cite{valjens:misb}]%\label{prop:ConnesM}
The Mi\v{s}\v{c}enko module $\mathrm{MF}$ is the finitely generated projective Hilbert $C^*$-module, described as the completion of $C_c(EG)$ with respect to the norm induced by the following $C^*_r(G)\,\hot\,C(BG) $-valued inner product:
\begin{equation}\label{eq:con1}
\bra{\xi}\ket{\zeta}(t)(x) = \sum_{p(y)=x} \bar{\xi}(y)\zeta(y \cdot t),
%
%\sum_{g\in G} \bra{\xi}\ket{\zeta}_g\otimes \la_g\label{eq:con1}\\\notag
%\bra{\xi}\ket{\zeta}_g(x)= ,
\end{equation}
where $\xi,\zeta\in C_c(EG)$, $t \in G$, $x \in BG$ and $p\colon EG\to BG$ is the quotient map. The right action of $C^*_r(G)\,\hot\,C(BG) $ on $M$ is defined by
\begin{equation}\label{eq:con2}
(\xi \cdot f)(y)=\sum_{g \in G} f(g)(p(y)) \cdot \xi(y \cdot g^{-1}),
\end{equation}
where $\xi\in C_c(EG)$, $f\in C_c(G, C(BG))$ and $y \in EG$.
\end{proposition}

We have for any separable $C^\ast$-algebra $B$ with trivial $G$-action \cite{markus:ass, valjens:misb},
\begin{align*}
\xymatrixcolsep{5pc}\xymatrix{
\KKK(C(BG), B) \ar[d]^{\cong}  \ar[r]^-{[\mathrm{MF}]\,\hot_{C(BG)}\,-} & \KKK(\mathbb{C}, C^\ast_r(G)\otimes B) \ar[d]^{=}   \\
\KKK^G(C_0(EG), B)   \ar[r]^-{\mu^G_B} & \KKK(\mathbb{C}, C^\ast_r(G)\otimes B)    
}
\end{align*}

The vertical isomorphism above is implemented by the strong Morita equivalence between $C(BG)$ and $C_0(EG)\rtimes G$ \cite{rieffel:morita}, whose associated $\KKK$-class is denoted $[Y^*]$ below (we use $[Y]$ for the opposite module).

%Let $(H, T)$ be a $G$-equivariant Kasparov cycle with property $(\gamma)$ and $x$ be the corresponding element in $\KKK^G(\mathbb{C}, \mathbb{C})$. 

If $G$ admits a compact non-positively curved manifold as a model for $BG$, then the element $\underline{d}$ was introduced by Kasparov \cite{kas:descent} as a ``dual-Dirac'' class 
\[
\underline{d}\in \KK{}{C^*_r(G)\hatotimes C(BG)}{\mathbb{C}}.
\]

To be more consistent with the terminology of this paper, $\underline{d}$ should be called the duality counit induced by the $\gamma$-element (which exists in this situation). Kasparov went on to show that $\underline{d}$ defines a left inverse for the assembly map.

Hence we see that we are in a situation where Spanier--Whitehead duality comes into play very naturally, with the choice $\mathrm{MF}=\text{unit}$ and $\underline{d}=\text{counit}$. Note that, while the class $\underline{d}$ requires structural information on the group, the class of the Mi\v{s}\v{c}enko bundle relies on very little structure. This is in complete analogy with the canonical unit defined previously.

\begin{proposition}
The class $\mathrm{MF}$ coincides with $\delta_G$ from Definition \ref{def:cunit} up to $\KKK$-equivalence. More precisely, we have
\begin{equation*}
\delta_G= \mathrm{MF}\,\hot_{C^*_r(G)\hot C(BG)}\, \tau_{C^*_r(G)}([Y^*]).
\end{equation*}
\end{proposition}
\begin{proof}
Let us set $Z=\mathrm{GJ}^{-1}([Y])\in\KK{G}{C_0(EG)}{C(BG)}$. It is shown in \cite{valjens:misb} that $Z$ is represented by a $G$-$C^*$-correspondence satisfying the following isomorphism of Hilbert modules
\[
\mathrm{MF}\cong i^*(Y^*)\,\hot_{C_0(EG)\rtimes G} \,(Z\rtimes_r G)
\]
(we are denoting by $i$ the inclusion $\C\hookrightarrow C(BG)$ as constant functions).
We want to prove
\[
[p_G]\,\hot_{C_0(EG)\rtimes G}\, [\Delta]=i^*([Y^*])\,\hot_{C_0(EG)\rtimes G}\, \jmath^G_r(Z) \,\hot_{C^*_r(G)\hot C(BG)}\, \tau_{C^*_r(G)}([Y^*]),
\]
or equivalently, by Lemma \ref{lem:desfact},
\begin{multline*}
[p_G]\,\hot_{C_0(EG)\rtimes G}\, [\Delta]=\\i^*([Y^*])\,\hot_{C_0(EG)\rtimes G}\, ([\Delta]\hatotimes\tau_{C^*_r(G)}(\mathrm{GJ}(Z))) \,\hot_{C^*_r(G)\hot C(BG)}\, \tau_{C^*_r(G)}([Y^*]).
\end{multline*}
It is well-known that $[p_G]=i^*([Y^*])$ (see for example \cite{markus:ass}), so that by associativity of the Kasparov product we have reduced the problem to showing
\[
\tau_{C^*_r(G)}(\mathrm{GJ}(Z))) \,\hot\, \tau_{C^*_r(G)}([Y^*])=\tau_{C^*_r(G)}(\mathrm{GJ}(Z)\,\hot_{C(BG)}\,[Y^*])=1_{C^*_r(G)\hot C_0(EG)\rtimes G}.
\]
Now $\mathrm{GJ}(Z)=[Y]$ by construction, hence the proof is complete.
\end{proof}

Now suppose that $G$ is a general torsion-free group, and that a $(\gamma)$-element $x=[H, T]$ exists. Inspired by Kasparov's construction, we define the class $\underline{d}$ in $\KKK(C^\ast_r(G)\otimes C(BG), \C)$ by setting
\[
\underline{d}= [Y]\,\hot_{C_0(\E)\rtimes G}\, d.
\]
The element $\underline{d}$ admits a simple description in terms of the cycle $(H, T)$ with property $(\gamma)$ as follows. The $G$-equivariant non-degenerate representation $\pi$ of $C_0(\E)$ on $H$ extends to the one of the multiplier algebra $C_b(\E)$. Together with the representation $\pi_G$ of $G$ on $H$, it induces the representation $\pi_G\otimes \pi$ of $C^*_r(G)\otimes C(BG)$ on $H$. Here, $C(BG)$ is naturally identified as the subalgebra $C_b(\E)$ consisting of $G$-invariant functions. The representation $\pi_G$ extends to the one for $C^\ast_r(G)$ since $\pi_G$ is weakly contained in the left regular representation. Indeed, $\pi_G$ is contained in the (amplified) left regular representation as we have a $G$-equivariant embedding from $H$ to $H\otimes \ell^2(G)$ given by
\[
v \quad \mapsto \quad \sum_h \pi(h(c))v\otimes\delta_h.
\]

\begin{proposition} The triple $(H, \pi_G\otimes \pi,T)$ defines a Kasparov cycle $[\pi_G\otimes \pi, H, T]$ for $\KKK(C^\ast_r(G)\otimes C(BG), \C)$. We have $[\pi_G\otimes \pi, H, T]=\underline{d}$.
\end{proposition}
\begin{proof} We need to show that for any $G$-invariant continuous function $\phi$ on $\E$, the commutator $[T, \phi]$ is compact. By the condition (2.2) for property $(\gamma)$, we just need show that $[T', \phi]$ is compact where $T'=\sum_{g\in G} g(c)Tg(c)$; $c$ is a cutoff function on $\E$. Take any compactly supported function $\chi$ on $\E$ so that $c\chi=c$. 

We have
\[
[T', \phi]=\sum_{g\in G} g(c)[T, g(\chi\phi)]g(c) = \sum_{g\in G} g(c)T_gg(c)
\]
where $T_g=[T, g(\chi\phi)]$ are compact operators whose norm vanish as $g$ goes to infinity by the condition (2.1) for property $(\gamma)$. It follows that $[T', \phi]=\sum_{g\in G} g(c)T_gg(c)$ is compact (see Lemma 2.5, 2.6 of \cite{nish:dsplitb}). 

We leave to the reader the straightforward check that the element $[H, \pi_G\otimes \pi, T]$ in $\KKK(C^\ast_r(G)\otimes C(BG), \C)$ corresponds to $d$ in $\KKK(C^\ast_r(G)\otimes C_0(\E)\rtimes G, \C)$ by the Morita equivalence between $C(BG)$ and $C_0(EG)\rtimes G$.
\end{proof}

We set 
\[
\underline\Lambda_{C^\ast_r(G)}=[\mathrm{MF}]\, \hot_{C(BG)}\, \underline d,\qquad \underline\Lambda_{C(BG)}=[\mathrm{MF}]\,\hot_{C^\ast_r(G)}\, \underline d.
\]

The following conclusions are immediate from the discussion above.
\begin{theorem} 
Let $G$ be a torsion-free group and suppose a $(\gamma)$-element $x\in \KK{G}{\C}{\C}$ exists. We have 
\[
\underline\Lambda_{C^\ast_r(G)}=\jmath^G_r(x),\qquad
\underline\Lambda_{C(BG)}=1_{C(BG)}.
\] 
For example, this is the case when $BG$ is a compact smooth manifold of non-postive sectional curvature.
\end{theorem}

\section{Examples}\label{sec:ex}

In this section we give a few examples and computations to put into context the abstract duality results that have been explained previously. We primarily treat the case of strong Spanier--Whitehead duality, and briefly discuss the weak case as it is mostly covered by other results in the literature (see for example \cite[Example 2.14 \& 2.17]{bvrs:string}).

\subsection{Groups with Spanier--Whitehead \texorpdfstring{$K$}{K}-duality} Let $G$ be a countable discrete group which satisfies the following two conditions (1), (2) or (1), (3):
\begin{enumerate}
\item[(1)] $G$ admits a $G$-compact model of $\E$;
\item[(2)] $G$ admits a $\gamma$-element $\gamma$ such that $\jmath^G_r(\gamma)=1_{C^\ast_r(G)}$, or 
\item[(3)] $G$ admits a $(\gamma)$-element $x$ such that $\jmath^G_r(x)=1_{C^\ast_r(G)}$.
\end{enumerate}
We recall that the gamma element $\gamma$, if exists, is represented by a cycle with property $(\gamma)$. Therefore, the condition (2) implies (3). Our previous argument shows that such a group $G$ has Spanier--Whitehead $K$-duality. Thanks to the Higson--Kasparov Theorem \cite{higkas:bc}, we obtain the following:
\begin{theorem} All a-T-menable groups which admit a $G$-compact model of $\E$ have Spanier--Whitehead $K$-duality.
\end{theorem}
Examples of such a-T-menable groups are the following:
\begin{itemize}
\item All groups which act properly, affine-isometrically and co-compactly on a finite-dimensional Euclidean space.
\item All co-compact lattices of simple Lie groups $\mathrm{SO}(n, 1)$ or $\mathrm{SU}(n,1)$.
\item All groups which act co-compactly on a tree (or more generally on a CAT(0)-cube complex).
\end{itemize}
For any a-T-menable group $G$ listed above, the gamma element $\gamma$ can be represented by an explicit cycle with property $(\gamma)$. Below, we describe an explicit cycle with property $(\gamma)$ for these groups. As a consequence, we can obtain an explicit cycle $d$ in $\KKK(C^\ast_r(G)\otimes C_0(\E)\rtimes G, \mathbb{C})$ which together with $\delta$, induces the duality between $C^\ast_r(G)$ and $C_0(\E)\rtimes G$. 

To begin with, we recall from \cite{kas:descent, valette:bc} that the gamma element exists for any group $G$ which acts properly, isometrically on a simply connected, complete Riemannian manifold $M$ of non-positive sectional curvature which is bounded from below. In this case, the gamma element for $G$ is represented by an unbounded $G$-equivariant Kasparov cycle 
\[
(H_M, D_M)
\] 
where $H_M$ is the Hilbert space $L^2(M, \Lambda^\ast T^\ast_{\mathbb{C}}M)$ of $L^2$-sections of the complexified exterior algebra bundles on $M$ and where $D_M$ is the self-adjoint operator 
\[
D_M=d_f+d_f^\ast
\]
on $M$ given by the following Witten type perturbation
\[
d_f= d+df\wedge
\]
of the exterior derivative $d$; the function $f$ is the squared distance $d_{M}^2(x_0, x)$ on $M$ for some fixed point $x_0$ of $M$. Let
\[
F_M= \frac{D_M}{(1+D_M^2)^{\frac12}}
\]
be the bounded transform of $D_M$. The element $[H_M, F_M]$ in $\KKK^G(\mathbb{C}, \mathbb{C})$ is the gamma element for $G$. We now suppose furthermore that $G$ action on $M$ is co-compact. In this case, $G$ admits a $G$-compact model of $\E$, namely the manifold $M$. 
\begin{proposition} The cycle $(H_M, F_M)$ has property $(\gamma)$.
\end{proposition}
\begin{proof} Since $[H_M, F_M]$ is the gamma element for $G$, it satisfies the condition (1) of property $(\gamma)$. To show that the condition (2) holds for $[H_M, F_M]$, we shall apply Theorem 6.1 of \cite{nish:dsplitb}. We use the natural non-degenerate representation of $C_0(M)$ on $H_M$ by pointwise multiplication. We take the dense subalgebra $B$ of $C_0(M)$ consisting of compactly supported smooth functions. Note that $B$ contains a cutoff function of $M$. For any function $h$ in $B$, we have
\[
[D_M, g(h)]=[d+d^\ast, g(h)] = g(c(h))  
\] 
where $c(h)$ is the Clifford multiplication by the gradient of $h$ which is bounded and compactly supported. We can now use Theorem 6.1 of \cite{nish:dsplitb} to conclude that the bounded transform $(H_M, F_M)$ satisfies the condition (2) of property $(\gamma)$.
\end{proof} 

\begin{corollary} For all groups $G$ which act properly, affine-isometrically, and co-compactly on a finite-dimensional Euclidean space $\mathbb{R}^n$, the $G$-equivariant cycle $(H_{\mathbb{R}^n}, F_{\mathbb{R}^n})$ has property $(\gamma)$.
\end{corollary}

\begin{corollary}\label{cor-Lie-gamma} For all co-compact closed subgroups $G$ of a semi-simple Lie group $L$,  the $G$-equivariant cycle $(H_{L/K}, F_{L/K})$ has property $(\gamma)$, where $K$ is a maximal compact subgroup of $L$.
\end{corollary}

Let us look at a few examples.

\begin{description}
\item[Poincar\'e--Langlands duality] In \cite{npw:langbc} the authors examine the Baum--Connes correspondence for the (extended) affine Weyl group $W_a$ associated to a compact connected semisimple Lie group $G$. This group can be realized as the group of affine isometries of the Lie algebra $\mathfrak{t}$ of a maximal torus $T\subseteq G$. The structure of $W_a$ is that of a semidirect product $\Gamma\rtimes W$, where $\Gamma$ is the lattice of translations in $\mathfrak{t}$, and $W$ is the Weyl group of the root system of~$G$. 

Ultimately, it is shown that the Baum--Connes conjecture (which holds in this case) is equivalent to $T$-duality for the aforementioned torus $T$ and the Pontryagin dual $\hat{\Gamma}$ of the lattice $\Gamma$. From the viewpoint of Lie groups, $\hat{\Gamma}$ equivariantly coincides with the maximal torus $T^{\vee}$ of the Langlands dual $G^{\vee}$ of $G$. In $K$-theory this is expressed by $W$-equivariant Spanier--Whitehead duality between the dual tori $T$ and $T^\vee$, which is referred to as ``Poincar\'e--Langlands duality'' in \cite{npw:langbc}. 

The results from Subsection \ref{subsec:tfree} in this paper can be equivalently applied to get these results, with $C(B\Gamma)$ playing the role of $C(T)$ and $C^*_r(\Gamma)$ playing the role of $C(T^{\vee})$ through the Gelfand transform. 

The $(\gamma)$-element, which belongs to $\KK{W_a}{\C}{\C}$, in this case can be constructed as explained above with $M=\mathfrak{t}$ and distance function induced by a $W$-equivariant metric. Equivalently, the bounded transform of the Bott--Dirac operator
\[
B_{\mathfrak{t}}=\sum_i (\mathrm{ext}(e_i)+ \mathrm{int}(e_1))x_i + (\mathrm{ext}(e_i) - \mathrm{int}(e_i))\frac{d}{dx_i}
\] 
yields a $W$-equivarant cycle with property $(\gamma)$, provided that interior multiplication is defined through a $W$-equivariant metric. The cycle obtained this way is indeed isomorphic to the one obtained through the Witten type perturbation of the de Rham operator, and its $\KKK$-class coincides with the classical $\gamma$-element which is homotopic to the unit \cite{higkas:bc}. 

In summary, we obtain equivariant duality classes $\underline{\delta}^W\in \KK{W}{\C}{C(\mathfrak{t}/\Gamma)\otimes C^*_r(\Gamma)}$, derived from the Mi\v{s}\v{c}enko $W$-bundle associated to the principal $\Gamma$-bundle $\mathfrak{t}\to T$, and $\underline{d}^W\in \KK{W}{C(\mathfrak{t}/\Gamma)\otimes C^*_r(\Gamma)}{\C}$, derived from the $(\gamma)$-element described above. We can prove
\[
\underline{\delta}^W\,\hot_{C(T)}\, \underline{d}^W = \jmath^\Gamma_r(\gamma),
\]
where on the right-hand side we mean ``partial'' descent with respect to the normal subgroup $\Gamma\subseteq W_a$.
As we know $\gamma=1_\C$ in $\KK{\Gamma\times W}{\C}{\C}$, so that we get respectively
\[
\underline{\delta}^W\,\hot_{C(T)}\, \underline{d}^W = 1,\qquad \underline{\delta}^W\,\hot_{C(T^\vee)}\,\underline{d}^W=1
\]
in the equivariant groups $\KK{W}{C(T^\vee)}{C(T^\vee)}, \KK{W}{C(T)}{C(T)}$.

\item[Lattices in $\mathrm{SO}(n, 1)$ and $\mathrm{SU}(n,1)$] Let $G$ be a co-compact lattice of a simple Lie group $L=\mathrm{SO}(n, 1)$, or $L=\mathrm{SU}(n,1)$. Let $K$ be a maximal compact subgroup of $L$.  Corollary \ref{cor-Lie-gamma} shows that the $G$-equivariant cycle $(H_{L/K}, F_{L/K})$ has property $(\gamma)$. The corresponding element $x=[H_{L/K}, F_{L/K}]$ is nothing but the gamma element $\gamma$ for $G$ which is shown to be equal to $1_G$ \cite{higkas:bc, kasjulg:su}.

\item[Groups acting on trees] Let $G$ be a countable discrete group which acts properly and co-compactly on a locally finite tree $Y$. The tree $Y$ is the union of the sets $Y^0$, $Y^1$ of the vertices and edges of the tree. Without loss of generality, we assume a $G$-invariant typing on the tree. Namely, we assume a $G$-invariant decomposition $Y^0=Y^0_0\sqcup Y^0_1$ so that any two adjacent vertices have distinct types. This can be achieved by the barycentric subdivision of the tree. We take $E$ as the geometric realization of the tree. This is a $G$-compact model of the universal proper $G$-space. We denote by $d$, the edge path metric on $E$ and hence on $Y^0$ so that  each edge has length $1$.

The $\ell^2$ space $\ell^2(Y)$ is naturally a graded $G$-Hilbert space with the even and odd spaces being $\ell^2(Y^0)$, $\ell^2(Y^1)$ respectively. Let $H_{\mathbb{R}}$ be the graded Hilbert space  $L^2(\mathbb{R}, \Lambda_{\mathbb{C}}^\ast(\mathbb{R}))$ as before, but now with the trivial $G$-action. We construct a Kasparov cycle with the property $(\gamma)$ on the graded tensor product
\[
H_Y =  H_{\mathbb{R}}\, \hot\, \ell^2(Y).
\]
Following \cite{kasskand:buildnov}, we define a non-degenerate representation $\pi$ of $C_0(E)$ on $H_Y$, which is diagonal with respect to $Y$. This is given by a family $(\pi_y)_{y\in Y}$ of representations of $C_0(E)$ on $H_{\mathbb{R}}$ indexed by $y$ in $Y$. If $y$ is a vertex, we define $\pi_y$ by sending $\phi$ in $C_0(E)$ to the multiplication on $H_\mathbb{R}$ by constant $\phi(y)$. If $y$ is an edge with vertices $y_0$, $y_1$ of corresponding types, we identify $y$ with the interval $[-\frac12, \frac12]$ via the unique isometry sending $y_j$ to $(-1)^{j}\frac12$.  We define $\pi_y$ by sending $\phi$ in $C_0(E)$ to the multiplication on $H_\mathbb{R}$ by the restriction of $\phi$ to the edge $y$ extended to left and right constantly.

Now, like the operator $D_M$, we shall define an unbounded, odd, self-adjoint operator $D_Y$ with compact resolvent of index $1$, which is almost $G$-equivariant and has nice compatibility with functions in $C_0(E)$. The bounded transform $F_Y$ of $D_Y$ will give us a desired Kasparov cycle $(H_Y, F_Y)$ with property $(\gamma)$. For this, we fix a base point $y_0$ from $Y^0$. The following construction depends on the choice of $y_0$. We have the following decomposition of $H_Y$: 
\[
H_Y= H_{\mathbb{R}}\,\hot\, \mathbb{C}\delta_{y_0} \oplus \bigoplus_{y\in Y^0 \backslash \{y_0\}} \left( H_{\mathbb{R}}\,\hot\,(\mathbb{C}\delta_y\oplus\mathbb{C}\delta_{e_y})    \right)
\]
where for each vertex $y\neq y_0$, $e_y$ is the last edge appearing in the geodesic from $y_0$ to $y$ and where the symbol $\delta_{\ast}$ denotes a delta-function in $\ell^2(Y)$. Our operator $D_Y$ is  block-diagonal with respect to this decomposition. It is given by a family $(D_y)_{y\in Y^0}$ of an unbounded, odd, self-adjoint operators with compact resolvent.

For a vertex $y\in Y^0_j$ of type $j$, let $B_{\mathbb{R}, y}$ be the Bott--Dirac operator on $H_\mathbb{R}$ with ``origin shifted'':
\[
B_{\mathbb{R},y}= (\mathrm{ext}(e_1)+ \mathrm{int}(e_1))(x-n_y) + (\mathrm{ext}(e_1) - \mathrm{int}(e_1))\frac{d}{dx}
\]
where $n_y=(-1)^j(\frac12+d(y, y_0))$. For $y=y_0$, we simply set 
\[
D_{y_0} = B_{\mathbb{R},y_0} \hot\,1 \,\,\,\,\,  \text{on \,\,\,\,\, $H_{\mathbb{R}}\,\hot \,\mathbb{C}\delta_{y_0}$}.
\]
For $y\neq y_0$, we set 
\[
D_y =B_{\mathbb{R},y}\hot\,1 + M_{\chi_y}\hot \left(
\begin{array}{cc}
0 & 1 \\
1 & 0 
\end{array}
\right)
 \,\,\,\,\,  \text{on \,\,\,\,\, $H_{\mathbb{R}}\,\hot\,(\mathbb{C}\delta_y\oplus\mathbb{C}\delta_{e_y})$}
\]
where $M_{\chi_y}$ is the multiplication on $H_{\mathbb{R}}$ by the function $\chi_y$ on $\mathbb{R}$ defined as:
\begin{equation*}
\text{for $y\in Y^{0}_0$,} \quad \chi_y(x) = \left \{ 
\begin{aligned}
&0 \,\,\,\,\, &{ x < \frac12} \\
&(x-\frac12)^2 \,\,\,\,\, &{\frac12 \leq x < 1}\\
&x-\frac34 \,\,\,\,\, &{ 1\leq x < d(y,y_0)}\\
&-(x-n_y)^2+d(y,y_0)-\frac12 \,\,\,\,\,&{ d(y,y_0) \leq x < n_y}\\
&d(y,y_0)-\frac12 \,\,\,\,\, &{ n_y\leq x },
\end{aligned}
\right.
\end{equation*}
\begin{equation*}
\text{for $y\in Y^{0}_1$,} \,\, \chi_y(x) = \left \{ 
\begin{aligned}
&d(y,y_0)-\frac12 \,\,\,\,\, &{ x< n_y\ }\\
&-(x-n_y)^2+d(y,y_0)-\frac12 \,\,\,\,\,&{ n_y\leq x < -d(y,y_0)}\\
&-x-\frac34 \,\,\,\,\, &{ -d(y,y_0) \leq y < -1}\\
&(x+\frac12)^2 \,\,\,\,\, &{-1\leq x < -\frac12}\\
&0 \,\,\,\,\, &{ -\frac12\leq x}. \\
\end{aligned}
\right.
\end{equation*}

Note that for each $y\neq y_0$, $D_y$ is a bounded perturbation of a self-adjoint operator $B_{\mathbb{R},y}\hot\,1$ with compact resolvent of index $0$, hence so is $D_y$. All $D_y$ are hence diagonalizable. Therefore, $D_Y=(D_y)_{y\in Y^0}$ is self-adjoint. In order to see that $D_Y$ has compact resolvent, we compute
\[
D_y^2 = B_{\mathbb{R},y}^2\hot\,1 + M_{\chi_y}^2\hot\,1 +  \left(
\begin{array}{cc}
0 & -M_{\chi_y'} \\
M_{\chi_y'} & 0 
\end{array}
\right)
\hot \left(
\begin{array}{cc}
0 & 1 \\
1 & 0 
\end{array}
\right)
\] 
where $\chi_y'$ is the derivative of $\chi_y$. We see that $D_y^2$ has spectrum far away from $0$ as $y$ goes to infinity essentially because the derivatives $\chi_y'$ are uniformly bounded in $y$ and because we have
\[ 
(x-n_y)^2 + \chi_y^2  \geq  2\Bigl(\frac{d(y, y_0)}{2}-\frac18\Bigr)^2 
\]
everywhere. It follows $D_Y$ has indeed, compact resolvent. Let $F_Y$ be the bounded transform
\[
F_Y= \frac {D_Y} {(1+D_Y^2)^{\frac12}}.
\]

\begin{proposition} \label{prop_JulgVal} A pair $(H_Y, F_Y)$ is a $G$-equivariant Kasparov cycle with property $(\gamma)$.
\end{proposition}
\begin{proof}
 Almost $G$-equivariance follows from
\[
D_Y - g(D_Y) = \text{bounded} \,\,\, \text{for $g \in G$}
\]
which we leave to the reader. To see that $[H_Y, F_Y]=1_F$ in $R(F)$ for any finite subgroup $F$ of $G$, we note that the class $[H_Y, F_Y]$ does not depend on the choice of the base point $y_0$. Hence, we may assume that $y_0$ is a vertex fixed by the group $F$. In this case, it is not hard to see that $F_Y$ is odd, $F$-equivariant, self-adjoint operator whose graded index is the one-dimensional trivial representation of $F$ spanned by $\xi_0\,\hot\,\delta_{y_0}$ in $H_{\mathbb{R}}\,\hot\, \mathbb{C}\delta_{y_0}$ where $\xi_0=e^{-\frac{x^2}{2}}$. This shows $[H_Y, F_Y]=1_F$. To show that is has the condition (2) of property $(\gamma)$ with respect to the representation $\pi$ of $C_0(E)$, we shall apply Theorem 6.1 of \cite{nish:dsplitb} for the dense subalgebra $B$ of $C_0(E)$ consisting of compactly supported functions which are smooth inside each edge and constant near the vertices. Note that $B$ contains a cutoff function of $E$. First, we can see that for each $y\neq y_0$, the operator $M_{\chi_y}\hot \left(
\begin{array}{cc}
0 & 1 \\
1 & 0 
\end{array}
\right)$ commutes with the representation $\pi$. This is due to the vanishing of $\chi_y$ for $y\in Y^0_0$ (resp. $Y^0_1$) over $x\leq\frac12$ (resp. over  $-\frac12\leq x$).  For $\phi$ in $B$, we compute the commutator $[D_Y, \pi(\phi)]$ as
\[
\begin{aligned}
{[D_Y, \pi(\phi)]}  &=   [B_{\mathbb{R},y_0}\hot\,1, \pi(\phi)] \\
&+  \sum_{y\in Y^{0}\backslash \{y_0\}}[B_{\mathbb{R},y}\hot\,1 + M_{\chi_y}\hot \left(
\begin{array}{cc}
0 & 1 \\
1 & 0 
\end{array} \right), \pi(\phi)]  \\
 & = [B_{\mathbb{R},y_0}\hot\,1, \pi(\phi)] + \sum_{y\in Y^{0}\backslash \{y_0\}}[B_{\mathbb{R},y}\hot\,1,  \pi(\phi)] \\
 &= \sum_{y\in Y^{0}\backslash \{y_0\}}[ \left(
\begin{array}{cc}
0 & -\frac{d}{dx} \\
\frac{d}{dx} & 0 
\end{array}
\right),  \pi_{e_y}(\phi)]\hot\,1 \\
 &=  \pi(\phi') \sum_{y\in Y^{0}\backslash \{y_0\}} \left(
\begin{array}{cc}
0 & -1 \\
1 & 0 
\end{array}
\right)\hot\,1
\end{aligned}
 \]
 where in the last two, each summand is an operator on $H_{\mathbb{R}}\hot\mathbb{C}\delta_{e_y}$ and where $\phi'$ is the derivative of $\phi$. Note that each summation is finite sum since $\phi$ is compactly supported. We can now use Theorem 6.1 of \cite{nish:dsplitb} to conclude that the bounded transform $(H_Y, F_Y)$ satisfies condition (2) in the definition of property $(\gamma)$.
\end{proof}

\begin{remark} The construction can be generalized to define a cycle with property $(\gamma)$ for a group which acts properly and co-compactly on a Euclidean building in a sense of \cite{kasskand:buildnov}. In \cite{brodguehignish:cat}, a different construction is given which provides us a cycle with property $(\gamma)$ for a group which acts properly and co-compactly on a finite-dimensional CAT(0) cube complex.
\end{remark}

\end{description}

\subsection{Groups with weak Spanier--Whitehead \texorpdfstring{$K$}{K}-duality} 
Let $G$ be a countable discrete group satisfying the two conditions (1), (2)' or (1), (3)' below:
\begin{enumerate}
\item[(1)] $G$ admits a $G$-compact model of $\E$;
\item[(2)'] $G$ admits a $\gamma$-element $\gamma$ with $\jmath^G_r(\gamma)$ acting as the identity on $K_\ast(C^\ast_r(G))$, or 
\item[(3)'] $G$ admits a $(\gamma)$-element $x$ with $\jmath^G_r(x)$ acting as the identity on $K_\ast(C^\ast_r(G))$.
\end{enumerate}
Our previous argument shows that such a group $G$ has weak Spanier--Whitehead $K$-duality. For any word-hyperbolic group, the gamma element is shown to exist and the Baum--Connes conjecture has been verified \cite{laff:hyp,kasskand:bolic,yumin:bchyp}. Moreover, any hyperbolic group is known to admit a $G$-compact model of $\E$ \cite{meinschick:hyp}. Hence, we have:
\begin{theorem} All word-hyperbolic groups $G$ have weak Spanier--Whitehead $K$-duality.
\end{theorem}
As an example of hyperbolic groups, we can take $G$ to be a co-compact lattices of the simple Lie group $L=\mathrm{Sp}(n, 1)$. As before, the $\gamma$-element for $G$ has an explicit representative $(H_{L/K}, F_{L/K})$ with property $(\gamma)$. We remark that the gamma element $\gamma=[H_{L/K}, F_{L/K}]$ is well-known to be not homotopic to $1_G$ due to Kazhdan's property $(\mathrm{T})$. Furthermore, Skandalis \cite{skand:knuc} showed that $\jmath^G_r(\gamma)$ is not equal to $1_{C^\ast_r(G)}$. More precisely, what he showed is that $C^\ast_r(G)$ is not $K$-nuclear, which in particular implies that it cannot be $\KKK$-equivalent to any nuclear $C^\ast$-algebra. The same remark that $\jmath^G_r(\gamma)\neq 1_{C^\ast_r(G)}$ applies to any infinite hyperbolic property $(\mathrm{T})$ group \cite[Theorem 5.2]{higgue:groupkth}. In general, when the gamma element $\gamma$ exists, the equality $\jmath^G_r(\gamma)=1_{C^\ast_r(G)}$ implies that $C^\ast_r(G)$ is $\KKK$-equivalent to $P_\C\rtimes G$, which satisfies the UCT (\cite[Proposition 9.5]{nestmeyer:loc}), in particular it is $K$-nuclear. Therefore, if $C^\ast_r(G)$ is not $K$-nuclear, we have $\jmath^G_r(\gamma)\neq 1_{C^\ast_r(G)}$.

\section{Some applications}\label{sec:cons}

In this section we prove a few results by applying the theory of $K$-duality developed in the previous pages. Some of the material presented here has been previously treated in the literature via possibly different methods \cite[Section 3]{dada:semiproj} \cite[Section 5]{meyereme:dualities}, \cite[Section 4.4]{wkp:ruellepoinc}, \cite[Section 7]{rosscho:kunneth}, nevertheless we provide a brief account for completeness, to give a better idea of some applications of our main theorems. %details for Corollaries \ref{cor:corD} and \ref{cor:corE}. 
%The reader may think of $A$ as $C^*_r(G)$ throughout this section. 

We say a $C^*$-algebra $A$ is \emph{$\KKK$-compact} if the functor sending $D$ to $\KK[*]{}{A}{D}$ commutes with {filtered} colimits. If $A$ is a $C^*$-algebra with a Spanier--Whitehead $K$-dual $B$, then $A$ is $\KKK$-compact because $\KK[*]{}{A}{D}$ is naturally isomorphic to $\KK[*]{}{\C}{D\otimes B}$ and the $K$-theory functor is continuous. 

As explained after Theorem 6.6 of \cite{nestmeyer:loc}, a $C^*$-algebra satisfies the \emph{Universal Coefficient Theorem} (UCT) (\cite[Section 23]{black:kth}) if and only if it belongs to the localizing triangulated subcategory of the $\KKK$-category generated by the complex numbers (this category is denoted as $\langle \ast \rangle$ in \cite{nestmeyer:loc}). As in \cite{DEKM:spectrum}, let us denote this subcategory by $\mathcal{T}$. It is known that within this subcategory, an object is dualizable if and only if it is compact:

\begin{proposition}(\cite[Proposition 4.1]{DEKM:spectrum})\label{prop:compdual}
In the subcategory $\mathcal{T} \subseteq \KKK$, the full triangulated subcategory $\mathcal{T}_c$ of compact objects coincides with the (closed) symmetric monoidal category $\mathcal{T}_d$ of dualizable objects. Furthermore, both these two subcategories are equal to the thick triangulated subcategory generated by the complex numbers. 
\end{proposition}

\begin{corollary}\label{cor_SWUCT} If $G$ has Spanier--Whitehead duality then $C^\ast_r(G)$ satisfies the UCT.
\end{corollary}
\begin{proof} We know that $C_0(\E)\rtimes G$ satisfies the UCT \cite[Proposition 9.5]{nestmeyer:loc}. By assumption, $C_0(\E)\rtimes G$ has a Spanier--Whitehead $K$-dual $C^\ast_r(G)$. Thus, $C_0(\E)\rtimes G$ is $\KKK$-compact. By Proposition \ref{prop:compdual}, it is dualizable in $\mathcal{T}$. Namely, it has a Spanier--Whitehead $K$-dual, say $A$, which satisfies the UCT. On the other hand, it is fairly easy to see that a dual object is unique up to equivalence. Hence, $C^\ast_r(G)$ is $\KKK$-equivalent to $A$. The claim follows from this.
\end{proof}

The strong Baum--Connes conjecture was introduced in \cite{nestmeyer:loc} as the assertion that the canonical Dirac morphism $\alpha$ in $\KKK^G(P_\C, \C)$ induces a $\KKK$-equivalence $\jmath^G_r(\alpha)$ from $P_\C\rtimes G$ to $C^*_r(G)$. In the presence of the gamma element $\gamma$ for $G$, this is equivalent to the assertion that $\jmath^G_r(\gamma)=1_{C^*_r(G)}$.

\begin{theorem}\label{thm:equivswdbc}
If $G$ has Spanier--Whitehead duality then it satisfies the strong Baum--Connes conjecture. Moreover, if the $\gamma$-element exists and $G$ satisfies the strong Baum--Connes conjecture, than $G$ has Spanier--Whitehead duality.
\end{theorem}
\begin{proof} 
Suppose $G$ has Spanier--Whitehead duality. Then, we know that the Baum--Connes conjecture holds for $G$, and so the Dirac morphism $\alpha$ induces an isomorphism $\jmath^G_r(\alpha)_*$ on $K$-groups from $P_\C\rtimes G$ to $C^\ast_r(G)$. Furthermore, both $P_\C\rtimes G$ and $C^\ast_r(G)$ satisfy the UCT by \cite[Proposition 9.5]{nestmeyer:loc} and by Corollary \ref{cor_SWUCT} respectively. It follows that $\jmath^G_r(\alpha)$ is a $\KKK$-equivalence \cite[Theorem 23.10.1]{black:kth}. Conversely, if the strong Baum--Connes conjecture holds, we have $\jmath^G_r(\gamma)=1_{C^*_r(G)}$. Hence, $G$ has Spanier--Whitehead duality by Theorem \ref{thm:swdm}.
\end{proof}

%However one soon notices that, while the right-hand side of \eqref{eq:assmapcoeffintro} admits a straightforward generalization by inserting a $C^*$-algebra in the left entry, the definition of the topological $K$-theory group does not have an equally obvious extension. This leads to considering two \emph{a priori} inequivalent variants \cite{otgon:bckk}. 
%The corollary above says that when the choice $A=C^*_r(G)$ is made, assuming duality, the subtleties vanish and ``one'' generalized definition of the conjecture is possible.

%In particular, thanks to Proposition \ref{thm:compdual} we see that $C^*$-algebras with a Spanier--Whitehead $K$-dual satisfy the UCT (localizing subcategories are automatically thick, see \cite{nee:tri}).

As in \cite[Theorem 23.10.5]{black:kth}, a $C^*$-algebra $A$ satisfies the UCT if and only if it is $\KKK$-equivalent to a commutative $C^*$-algebra $C_0(X)$. Furthermore, this $X$ can be taken to be a 3-dimensional cell complex (see \cite[Corollary 23.10.3]{black:kth}, \cite[Proposition 7.4]{rosscho:kunneth}). This is because the range of $K$-theory on such spaces exhausts all countable $\Z/(2)$-graded abelian groups. If $K_*(A)$ is finitely generated, then $X$ can be chosen finite, and a Spanier-Whitehead $K$-dual exists for such spaces \cite[Proposition 5.9]{meyereme:dualities}. %The following proof is well-known and it is included for completeness.

\begin{lemma}\label{lem:cptuctfin}
Suppose $A$ has a Spanier--Whitehead $K$-dual and satisfies the UCT. Then it has finitely generated $K$-theory groups.
\end{lemma}
\begin{proof}
As in the proof of \cite[Proposition 7.4]{rosscho:kunneth}, \cite[Corollary 23.10.3]{black:kth}, let $C=C^{0}\oplus C^{1}$ be a commutative $C^*$-algebra $\KKK$-equivalent to $A$, where $C^{0}$ is the mapping cone of a $\ast$-homomorphism on direct sums of $C_0(\R)$, and $C^{1}$ is the suspension of such a mapping cone. It is easy to see that $C$ is the inductive limit of subalgebras $C_n$ where $C_n$ has finitely generated $K$-theory. Since $\KK[*]{}{A}{-}$ is continuous (since $A$ is $\KKK$-compact), the equivalence $A\to C$ factors through $C_n$ for some $n\in \N$. Then $K_*(A)$ is finitely generated because it is a quotient of $K_*(C_n)$, which enjoys this property.
\end{proof}

%A comment on the previous corollary when $A=C^*_r(G)$: when $\gamma$ equals $1_{C^*_r(G)}$ after descent, then $C^*_r(G)$ automatically satisfies the UCT. This is because $P_\C\rtimes G$ satisfies the UCT \cite[Proposition 9.5]{nestmeyer:loc}. With this observation, the assumption is only relevant when the duality is obtained through the $(\gamma)$-element.

%Skandalis \cite{skand:knuc} has shown the existence of discrete groups for which we have $\jmath^G_r(\gamma)\neq 1$, even though the Baum--Connes conjecture holds. Our approach for these groups does not allow deducing (strong) Spanier--Whitehead duality, wherefore the question of finite generation of $K$-groups for such examples. It turns out we can answer this positively, which is a slight improvement on the previous result.

\begin{proposition}
Suppose $G$ satisfies the Baum--Connes conjecture and the $\gamma$-element exists. Then $C^*_r(G)$ has finitely generated $K$-theory groups.
\end{proposition}
\begin{proof}
{If $\gamma\in \KK{G}{\C}{\C}$ exists, then $P_\C\rtimes G$ is dualizable by Theorem \ref{thm:pdual}. It is known that $P_\C\rtimes G$ satisfies the UCT (see \cite[Proposition 9.5]{nestmeyer:loc}). Thus, $P_\C\rtimes G$ has finitely generated $K$-groups by Lemma \ref{lem:cptuctfin}.} Recall that in the localization picture the assembly map appears as
\begin{equation}\label{eq:assmaploc}
K_*(P_\C\rtimes G) \longrightarrow K_*(C^*_r(G)).
\end{equation}
Therefore if \eqref{eq:assmaploc} is an isomorphism the right-hand side is finitely generated.
\end{proof}
\begin{remark} {More generally, $C^*_r(G)$ has finitely generated $K$-theory groups if $G$ satisfies the Baum--Connes conjecture and the source $P_\C$ of the Dirac morphism is a (categorical) direct summand of a proper algebra. This is because by Remark \ref{rem_Kasparovdual} $P_\C\rtimes G$ has a Spanier--Whitehead $K$-dual.}
\end{remark}
\begin{remark}\label{rem:spectra}{
By the results in \cite{DEKM:spectrum}, there exists a functor $\mathbb{K}$ from the $\KKK$-category to the stable homotopy category, satisfying $\pi_n(\mathbb{K}(A))\cong K_n(A)$. This functor specializes to a full and faithful functor on the subcategory of dualizable objects satisfying the UCT, realizing $C^*$-algebras as perfect $\mathrm{KU}$-modules (in particular, finite spectra). Hence the previous results can be also obtained from the well-known fact that homotopy groups are finitely generated in this context.}
\end{remark}

Define the \emph{$n$-th dimension-drop algebra} as
\begin{equation*}
\I_n=\{ f\in C([0,1],M_n(\C)) \mid f(0)=0,f(1)\in \C 1_n\}.
\end{equation*}
We can use this to introduce the mod-$n$ $K$-theory groups as follows:
\begin{equation*}
K_*(B; \Z/(n))=\KK[*]{}{\I_n}{D}.
\end{equation*}

It is apparent from this definition that a Baum--Connes conjecture in mod-$n$ $K$-theory for $B$ would have to introduce coefficients on the left, and we can take this as motivation to find a satisfactory formulation for the full bivariant version of the Baum--Connes conjecture. The approach via localization immediately generalizes to this context, giving us a map
\begin{equation}\label{eq:bivassmap}
\KK[*]{}{A}{(P_{\C}\otimes B)\rtimes G}\longrightarrow \KK[*]{}{A}{B\rtimes_r G}
\end{equation}
defined as $y\mapsto y \otimes \jmath^G_r(1_B\hatotimes \alpha)$, where $\alpha\in\KK{}{P_{\C}}{\C}$ is the Dirac morphism, for any (separable) $C^*$-algebra $A$ and $G$-$C^*$-algebra $B$.

The original definition of the left-hand side (following \cite{bmp:lgbc} and \cite{otgon:bckk}), what is called the ``naive'' topological $K$-group in \cite{otgon:bckk}, is given as 
\begin{equation*}
\varinjlim_{Y\subseteq \E} \KK[*]{G}{C_0(Y,A)}{B},
\end{equation*}
where the limit ranges as usual over $G$-invariant $G$-compact subspaces of $\E$. Unlike the simpler case of the conjecture, the definition making use of the naive topological group is \emph{not} equivalent to the definition in \eqref{eq:bivassmap}. However \cite{otgon:bckk} shows that there are natural maps
\begin{equation}\label{eq:nucomp}
\nu_Y\colon \KK[*]{G}{C_0(Y,A)}{B}\longrightarrow \KK[*]{}{A}{(P_{\C}\otimes B)\rtimes G},
\end{equation}
which make the obvious diagram commute. In addition, if $A$ admits a Spanier--Whitehead $K$-dual, then \eqref{eq:nucomp} induces an isomorphism. 

\begin{theorem}[\cite{otgon:bckk}]
{Suppose $A$ has a Spanier--Whithead $K$-dual}. Then the comparison map induced by the $\nu_Y$'s is an isomorphism.
\end{theorem}

%\appendix
%\input{Sections/dolor}
% *****************************************************************
% Ending
%******************************************************************
%\input{BeginEnd/Bibliography}

%\nocite{*}
%\printbibliography

\bibliography{BibliographyBST}
\bibliographystyle{amsalpha-init}
\end{document}